\documentclass[11pt, twoside]{article}
\usepackage{latexsym}
\usepackage{amsmath}
\usepackage{amssymb}
\usepackage[all]{xy}
\usepackage{amsfonts}
\usepackage{verbatim}
\usepackage{amsthm}
\usepackage{mathrsfs}
\usepackage{epsfig}
\usepackage{xy}
\usepackage{array}
\usepackage{stmaryrd}
\usepackage{graphicx,color}
\usepackage{xcolor}
\usepackage[colorlinks=true,linkcolor=blue,citecolor=blue]{hyperref}
\usepackage{tikz}
\usetikzlibrary{arrows,calc}
\usepackage{etex}
\usepackage{mathdots}
\usepackage{float}
\usepackage{graphics}
\usepackage{pdflscape}
\usepackage{tikz-cd}

\usepackage{anysize,hyperref}
\input xypic
\xyoption{all}
\usepackage{enumitem}
\usepackage[perpage,symbol]{footmisc}
\usepackage{setspace}
\renewcommand{\cal}{\mathcal}
\def\A{\mathcal{A}}

\def\dr{\ar@{->}[r]}

\usepackage{bm}
\begin{document}
\baselineskip=15pt
\title{\Large{\bf Triangulated categories arising from  $n$-fold matrix factorizations}
\footnotetext{Panyue Zhou is supported by the National Natural Science Foundation of China (Grant No. 12371034) and by the Scientific Research Fund of Hunan Provincial Education Department (Grant No. 24A0221).}}
\medskip
\author{Yixia Zhang and Panyue Zhou }

\date{}

\maketitle
\def\blue{\color{blue}}
\def\red{\color{red}}

\newtheorem{theorem}{Theorem}[section]
\newtheorem{lemma}[theorem]{Lemma}
\newtheorem{corollary}[theorem]{Corollary}
\newtheorem{proposition}[theorem]{Proposition}
\newtheorem{conjecture}{Conjecture}
\theoremstyle{definition}
\newtheorem{definition}[theorem]{Definition}
\newtheorem{question}[theorem]{Question}
\newtheorem{remark}[theorem]{Remark}
\newtheorem{remark*}[]{Remark}
\newtheorem{example}[theorem]{Example}
\newtheorem{example*}[]{Example}
\newtheorem{condition}[theorem]{Condition}
\newtheorem{condition*}[]{Condition}
\newtheorem{construction}[theorem]{Construction}
\newtheorem{construction*}[]{Construction}

\newtheorem{assumption}[theorem]{Assumption}
\newtheorem{assumption*}[]{Assumption}

\baselineskip=17pt
\parindent=0.5cm
\vspace{-6mm}

\begin{abstract}
\baselineskip=16pt
Let $\mathcal{A}$ be an additive category and let $T\colon \mathcal{A}\rightarrow \mathcal{A}$ be an additive functor equipped with a natural transformation $\omega\colon \mathrm{Id}_{\mathcal{A}}\rightarrow T$. We prove that the homotopy category of $n$-fold matrix factorizations of $\omega$, denoted ${\rm HFact}_{n}(\mathcal{A},T,\omega)$, admits a natural structure of a right triangulated category. In particular, when $T$ is an automorphism, the homotopy category ${\rm HFact}_{n}(\mathcal{A},T,\omega)$ becomes triangulated.
Furthermore, if $\mathcal{A}$ is a Frobenius exact category and $T$ is an autoequivalence, we obtain that the category ${\rm Fact}_{n}(\mathcal{A},T,\omega)$ of $n$-fold $(\mathcal{A},T)$-factorizations of $\omega$ is a Frobenius exact category.
Consequently, the stable category of the Frobenius exact category ${\rm Fact}_{n}(\mathcal{A},T,\omega)$ is a triangulated category.
\\[0.2cm]
\end{abstract}


\vspace{1cm}

\section{Introduction}

Since Eisenbud \cite{E} introduced matrix factorizations in 1980 to study free resolutions over hypersurface singularities in commutative algebra, the theory has developed in many directions. Matrix factorizations have appeared in the study of stable module categories \cite{B}, singularity categories \cite{O1}, quiver representations \cite{KST}, and a variety of geometric and categorical contexts \cite{O2,EP,P}. These developments demonstrate the broad influence of matrix factorizations within representation theory and algebraic geometry.

In order to extend the theory to noncommutative settings, Cassidy, Conner, Kirkman and Moore \cite{CCKM} introduced twisted matrix factorizations. Later, Bergh and Erdmann \cite{BE} studied twisted matrix factorizations for quantum complete intersections of codimension two. Their work generalizes classical matrix factorizations to arbitrary rings. In many respects, the twisted factorizations constructed by  Cassidy and collaborators and those developed by Bergh and Erdmann complement one another, and their constructions agree with the notion of module factorizations introduced by Chen \cite{C2}. Sun and Zhang \cite{SZ} further generalized this framework by introducing the theory of $n$-fold module factorizations. Their work unifies both the commutative theory of $n$-fold matrix factorizations and the noncommutative theory of module factorizations, and establishes triangle equivalences among the relevant stable categories.

From a more categorical point of view, Ballard, Deliu, Favero, Isik and Katzarkov \cite{BDFIK} introduced the idea of factorization categories. Their framework extends the concept of matrix factorizations to arbitrary abelian categories, and creates strong connections with derived categories and categorical symmetry. In recent work, Bergh and J{\o}rgensen \cite{BJ} introduced the category of $n$-fold matrix factorizations associated with a natural transformation for any even integer $n$. Their construction contains both the classical matrix factorizations introduced by Eisenbud and the two-fold factorizations studied in the setting of factorization categories.

When $n$ is an even integer and the functor $T$ is an automorphism, Bergh and J{\o}rgensen
\cite{BJ} proved that the homotopy category of $n$-fold matrix factorizations of a natural transformation carries the structure of an algebraic triangulated category. Their proof makes use of the viewpoint of pretriangulated differential graded categories. This raises a natural question: whether a similar conclusion remains valid when $T$ is replaced by an arbitrary additive functor.

In this paper we answer this question in the affirmative.
We show that the homotopy category of $n$-fold matrix factorizations admits a right triangulated structure. Let $\mathcal{A}$ be an additive category,  $T\colon \mathcal{A}\rightarrow \mathcal{A}$ be an additive functor, and let $\omega\colon {\rm Id}_{\mathcal{A}}\rightarrow T$ be a natural transformation. The data $(\mathcal{A},T,\omega)$ gives rise to the category
${\rm Fact}_{n}(\mathcal{A},T,\omega),$
whose objects are $n$-fold $(\mathcal{A},T)$-factorizations of $\omega$, that is, sequences
$$
X^{0}\xrightarrow{d_{X}^{0}}X^{1}\xrightarrow{d_{X}^{1}}\cdots\xrightarrow{d_{X}^{\,n-1}}T(X^{0})
$$
whose composite equals $\omega_{X^{0}}$, together with morphisms defined levelwise and compatible with the differentials.
We prove that when $T$ is additive, the associated suspension functor
$$
\Sigma\colon {\rm HFact}_{n}(\mathcal{A},T,\omega)\longrightarrow {\rm HFact}_{n}(\mathcal{A},T,\omega)
$$
induced by rotation of the factorization is itself additive. Consequently, the homotopy category
$$
{\rm HFact}_{n}(\mathcal{A},T,\omega)
$$
admits a natural structure of a right triangulated category, where the distinguished right triangles are given by mapping cones of morphisms of $n$-fold factorizations.
When $\mathcal A$ is a Frobenius exact category and, in addition, $T$ is an autoequivalence, we further investigate the exact structure on ${\rm Fact}_{n}(\mathcal{A},T,\omega)$.
By adapting and extending the arguments of Chen in the case of $2$-fold matrix factorizations~\cite{C}, we show that ${\rm Fact}_{n}(\mathcal{A},T,\omega)$ carries a Frobenius exact structure in which the projective-injective objects are precisely the contractible $n$-fold factorizations. As a consequence, the stable category of the Frobenius exact category $({\rm Fact}_{n}(\mathcal{A},T,\omega),\widetilde{\mathcal{E}})$ is a triangulated category.

The paper is organized as follows. In Section~2, we recall the necessary background on $n$-fold factorizations. In Section~3, we establish the right triangulated structure on the homotopy category ${\rm HFact}_{n}(\mathcal{A},T,\omega)$ and verify the required axioms directly from the definitions. In Section~4, we study the exact structure on ${\rm Fact}_{n}(\mathcal{A},T,\omega)$ and prove that it is a Frobenius exact category whenever $\mathcal A$ is a Frobenius exact category and $T$ is an autoequivalence.

\section{Preliminaries}
	
Let $\mathcal{A}$ be an additive category and let $T\colon\mathcal{A}\rightarrow\mathcal{A}$ be an additive functor.
Fix a natural transformation $\omega:\mathrm{Id}_{\mathcal{A}}\rightarrow T$ satisfying
$\omega_{T(X)} = T(\omega_X)$ for every object $X \in \mathcal{A}$.
We begin by recalling the notion of an $n$-fold $(\mathcal{A},T)$-factorization $(X^{j}, d_{X}^{j})$ of $\omega$,
as introduced in \cite{BJ}.

\begin{definition}
Let $n \ge 2$ be a positive integer.
An $n$-fold $(\mathcal{A},T)$-factorization
$X = (X^{j}, d_{X}^{j})$ of $\omega$ is a sequence
\begin{center}
$X^{0} \stackrel{d_{X}^{0}}{\longrightarrow} X^{1}
   \stackrel{d_{X}^{1}}{\longrightarrow} X^{2}
   \stackrel{d_{X}^{2}}{\longrightarrow} X^{3}
   \longrightarrow \cdots
   \stackrel{d_{X}^{n-2}}{\longrightarrow} X^{\,n-1}
   \stackrel{d_{X}^{\,n-1}}{\longrightarrow} T(X^{0}),$
\end{center}
in which every $n$-fold composition equals the corresponding component of $\omega$:
\begin{align*}
d_{X}^{\,n-1} \circ d_{X}^{\,n-2} \circ \cdots \circ d_{X}^{1} \circ d_{X}^{0}
  &= \omega_{X^{0}}, \\[2mm]
T(d_{X}^{0}) \circ d_{X}^{\,n-1} \circ d_{X}^{\,n-2} \circ \cdots \circ d_{X}^{2} \circ d_{X}^{1}
  &= \omega_{X^{1}}, \\[2mm]
T(d_{X}^{1}) \circ T(d_{X}^{0}) \circ d_{X}^{\,n-1} \circ \cdots \circ d_{X}^{3} \circ d_{X}^{2}
  &= \omega_{X^{2}}, \\[2mm]
&\hspace{-2cm}\vdots \\[2mm]
T(d_{X}^{\,n-2}) \circ T(d_{X}^{\,n-3}) \circ \cdots \circ T(d_{X}^{0}) \circ d_{X}^{\,n-1}
  &= \omega_{X^{\,n-1}}.
\end{align*}

\end{definition}

A morphism
$\varphi\colon (X^{j}, d_{X}^{j}) \rightarrow (Y^{j}, d_{Y}^{j})$
between two $n$-fold $(\mathcal{A},T)$-factorizations of $\omega$
is a sequence $(\varphi^{0},\varphi^{1},\dots,\varphi^{n-1})$ of morphisms in $\mathcal{A}$
such that the following diagram commutes:
\[
\xymatrix{
X^{0}\ar[r]^{d_{X}^{0}}\ar[d]_{\varphi^{0}}
  & X^{1}\ar[r]^{d_{X}^{1}}\ar[d]_{\varphi^{1}}
  & X^{2}\ar[r]^{d_{X}^{2}}\ar[d]_{\varphi^{2}}
  & \cdots\ar[r]^{d_{X}^{\,n-2}}
  & X^{\,n-1}\ar[r]^{d_{X}^{\,n-1}}\ar[d]_{\varphi^{\,n-1}}
  & T(X^{0})\ar[d]^{T(\varphi^{0})} \\
Y^{0}\ar[r]^{d_{Y}^{0}}
  & Y^{1}\ar[r]^{d_{Y}^{1}}
  & Y^{2}\ar[r]^{d_{Y}^{2}}
  & \cdots\ar[r]^{d_{Y}^{\,n-2}}
  & Y^{\,n-1}\ar[r]^{d_{Y}^{\,n-1}}
  & T(Y^{0}).
}
\]
Composition of morphisms is defined componentwise,
and this endows the collection of all $n$-fold $(\mathcal{A},T)$-factorizations of $\omega$
with the structure of a category, denoted by
${\rm Fact}_{n}(\mathcal{A},T,\omega)$.
Since $\mathcal{A}$ is additive and $T$ is an additive functor,
the category ${\rm Fact}_{n}(\mathcal{A},T,\omega)$ is also additive.

\begin{remark}
\begin{itemize}
    \item[(1)]
    A morphism
    $\varphi = (\varphi^{0},\varphi^{1},\dots,\varphi^{n-1})$
    in $\mathrm{Fact}_{n}(\mathcal{A},T,\omega)$
    is an isomorphism if and only if each $\varphi^{j}~(j=0,1,\cdots,n-1)$ is an isomorphism in $\mathcal{A}$.

    \item[(2)]
    The zero object of $\mathrm{Fact}_{n}(\mathcal{A},T,\omega)$
    is the trivial $n$-fold $(\mathcal{A},T)$-factorization $(0,0)$ of $\omega$.
    Explicitly,
    \[
    (0,0)=
        0 \stackrel{0}{\longrightarrow} 0
        \stackrel{0}{\longrightarrow} 0
        \longrightarrow \cdots
        \stackrel{0}{\longrightarrow} 0
        \stackrel{0}{\longrightarrow} 0,
    \]
    consisting of $n$ consecutive zero morphisms.

    \item[(3)]
    The coproduct of two $n$-fold $(\mathcal{A},T)$-factorizations of $\omega$
    is taken componentwise.
    Since
    \[
    \omega_{X\oplus Y}=
    \begin{bmatrix}
        \omega_{X} & 0 \\
        0 & \omega_{Y}
    \end{bmatrix},
    \]
    the result again forms an $n$-fold $(\mathcal{A},T)$-factorization of $\omega$.
For any $X,Y \in \mathcal{A}$, the coproduct $X \oplus Y$ is given by
\begin{flushleft}
$ X^{0}\oplus Y^{0}
   \xrightarrow{
   \begin{bmatrix}
      d_{X}^{0} & 0 \\
      0 & d_{Y}^{0}
   \end{bmatrix}}
   X^{1}\oplus Y^{1}
   \xrightarrow{
   \begin{bmatrix}
      d_{X}^{1} & 0 \\
      0 & d_{Y}^{1}
   \end{bmatrix}}
   X^{2}\oplus Y^{2}
   \longrightarrow \cdots$
\end{flushleft}

\begin{flushright}
$\cdots \longrightarrow
   X^{n-2}\oplus Y^{n-2}
   \xrightarrow{
   \begin{bmatrix}
      d_{X}^{n-2} & 0 \\
      0 & d_{Y}^{n-2}
   \end{bmatrix}}
   X^{n-1}\oplus Y^{n-1}
   \xrightarrow{
   \begin{bmatrix}
      d_{X}^{n-1} & 0 \\
      0 & d_{Y}^{n-1}
   \end{bmatrix}}
   T(X^{0})\oplus T(Y^{0}).$
\end{flushright}
\end{itemize}
\end{remark}

\begin{definition}
Let $X, Y \in \mathrm{Fact}_{n}(\mathcal{A},T,\omega)$, and let
$$f = (f^{0}, f^{1}, \dots, f^{n-1}),\
 g = (g^{0}, g^{1}, \dots, g^{n-1})\colon X \rightarrow Y$$
be morphisms.
We say that $f$ is \emph{homotopic} to $g$, written $f \sim g$,
if there exist morphisms
$s = (s^{0}, s^{1}, \dots, s^{n-1})$ in $\mathcal{A}$, called \emph{diagonal morphisms},
such that the following diagram commutes:
\[
\centerline{
\xymatrix{
X^{0} \ar[r]^{d_{X}^{0}} \ar[d]
    & X^{1} \ar[r]^{d_{X}^{1}} \ar@{.>}[ld]_{s^{0}} \ar[d]
    & X^{2} \ar[r]^{d_{X}^{2}} \ar@{.>}[ld]_{s^{1}} \ar[d]
    & \cdots \ar[r]^{d_{X}^{n-2}}
    & X^{n-1} \ar[r]^{d_{X}^{n-1}} \ar@{.>}[ld]_{s^{n-2}} \ar[d]
    & T(X^{0}) \ar@{.>}[ld]_{s^{n-1}} \ar[d]
    \\
Y^{0} \ar[r]^{d_{Y}^{0}}
    & Y^{1} \ar[r]^{d_{Y}^{1}}
    & Y^{2} \ar[r]^{d_{Y}^{2}}
    & \cdots \ar[r]^{d_{Y}^{n-2}}
    & Y^{n-1} \ar[r]^{d_{Y}^{n-1}}
    & T(Y^{0})
}}
\]
and the following equalities hold:
\begin{align*}
T(f^{0}-g^{0}) &= d_{Y}^{n-1} \circ s^{\,n-1} + T(s^{0} \circ d_{X}^{0}), \\
f^{1}-g^{1} &= d_{Y}^{0} \circ s^{0} + s^{1} \circ d_{X}^{1}, \\
&~\vdots \\
f^{\,n-1}-g^{\,n-1} &= d_{Y}^{\,n-2} \circ s^{\,n-2} + s^{\,n-1} \circ d_{X}^{\,n-1}.
\end{align*}
\end{definition}

This defines an equivalence relation on the set of morphisms from $X$ to $Y$.
The equivalence class of a morphism $f$ is denoted by $[f]$.
If $f_{1}, f_{2}, g_{1}, g_{2} : X \rightarrow Y$ are morphisms in
$\mathrm{Fact}_{n}(\mathcal{A},T,\omega)$
with $f_{1} \sim g_{1}$ and $f_{2} \sim g_{2}$,
then $f_{1} + f_{2} \sim g_{1} + g_{2}$.
Thus homotopy is compatible with addition, and the sum of classes satisfies
$
[f_{1}] + [f_{2}] = [f_{1} + f_{2}].
$
Similarly, if $f_{1}, f_{2}: X \rightarrow Y$ and
$g_{1}, g_{2}: Y \rightarrow Z$ satisfy
$f_{1} \sim f_{2}$ and $g_{1} \sim g_{2}$,
then
\[
g_{1} \circ f_{1} \sim g_{2} \circ f_{2}.
\]
Hence homotopy is compatible with composition, and the induced composition on classes,
\[
[g] \circ [f] := [\,g \circ f\,],
\]
is well defined.
We therefore obtain the \emph{homotopy category}
$
\mathrm{HFact}_{n}(\mathcal{A},T,\omega),
$
whose objects are those of $\mathrm{Fact}_{n}(\mathcal{A},T,\omega)$
and whose morphisms are the homotopy classes of morphisms in
$\mathrm{Fact}_{n}(\mathcal{A},T,\omega)$.

\section{$\text{HFact}_{n}(\mathcal{A},T,\omega)$ is a right triangulated category}

In this section, we assume that $n$ is an even integer.

\begin{definition}
The suspension $\Sigma X$ of an $n$-fold $(\mathcal{A},T)$-factorization $X$ of $\omega$ is the sequence
	\begin{center}
		$X^{1} \xrightarrow{-d_{X}^{1}} X^{2} \xrightarrow{-d_{X}^{2}} X^{3}\longrightarrow \cdots \longrightarrow X^{n-1} \xrightarrow{-d_{X}^{n-1}} T(X^{0}) \xrightarrow{-T(d_{X}^{0})} T(X^{1})$.
	\end{center}
\end{definition}

\begin{remark}
\begin{itemize}
    \item[(1)]
    The definition above is well defined.
    Since $n$ is even, all signs occurring in an $n$-fold composition cancel,
    and the resulting morphisms satisfy the required identities.

    \item[(2)]
    Any morphism $f\colon X \rightarrow Y$ induces a morphism
    \[
        \Sigma f = (f^{1}, f^{2}, \dots, T(f^{0})) :
        \Sigma X \rightarrow \Sigma Y,
    \]
    which makes the diagram
    \[
    \xymatrix@C=1.3cm{
        X^{1} \ar[r]^{-d_{X}^{1}} \ar[d]^{f^{1}}
            & X^{2} \ar[r]^{-d_{X}^{2}} \ar[d]^{f^{2}}
            & X^{3} \ar[r]^{-d_{X}^{3}} \ar[d]^{f^{3}}
            & \cdots \ar[r]^{-d_{X}^{n-1}}
            & T(X^{0}) \ar[r]^{-T(d_{X}^{0})} \ar[d]^{T(f^{0})}
            & T(X^{1}) \ar[d]^{T(f^{1})}
        \\
        Y^{1} \ar[r]^{-d_{Y}^{1}}
            & Y^{2} \ar[r]^{-d_{Y}^{2}}
            & Y^{3} \ar[r]^{-d_{Y}^{3}}
            & \cdots \ar[r]^{-d_{Y}^{n-1}}
            & T(Y^{0}) \ar[r]^{-T(d_{Y}^{0})}
            & T(Y^{1})
    }
    \]
    commute.

    \item[(3)]
    The suspension functor preserves homotopies, coproducts, addition,
    and composition of morphisms.
    Consequently, it induces an additive endofunctor
    \[
        {\rm HFact}_{n}(\mathcal{A},T,\omega)
            \xrightarrow{\;\Sigma\;}
        {\rm HFact}_{n}(\mathcal{A},T,\omega).
    \]
\end{itemize}
\end{remark}

\begin{definition}
For a morphism $f : X \rightarrow Y$ in
${\rm Fact}_{n}(\mathcal{A},T,\omega)$,
the \emph{mapping cone} of $f$ is defined as
\[
    C_{f} := \Sigma X \oplus Y,
\]
given by the sequence
\[
X^{1}\oplus Y^{0}
   \xrightarrow{d_{C_{f}}^{0}}
X^{2}\oplus Y^{1}
   \xrightarrow{d_{C_{f}}^{1}}
X^{3}\oplus Y^{2}
   \longrightarrow \cdots
   \longrightarrow
T(X^{0})\oplus Y^{n-1}
   \xrightarrow{d_{C_{f}}^{\,n-1}}
T(X^{1})\oplus T(Y^{0}),
\]
where
\[
d_{C_{f}}^{j} =
\begin{bmatrix}
    -d_{X}^{\,j+1} & 0 \\
    f^{\,j+1}      & d_{Y}^{\,j}
\end{bmatrix}
\quad (0 \le j \le n-2),
\qquad
d_{C_{f}}^{\,n-1} =
\begin{bmatrix}
    -T(d_{X}^{0}) & 0 \\
    T(f^{0})      & d_{Y}^{\,n-1}
\end{bmatrix}.
\]
\end{definition}

We verify that this construction yields an $n$-fold
$(\mathcal{A},T)$-factorization of $\omega$.
Indeed, since the integer $n$ is even,
the signs appearing in the definition of $C_{f}$ cancel appropriately,
and the construction is therefore well defined.
Naturally, there are two morphisms in ${\rm Fact}_{n}(\mathcal{A},T,\omega)$:
\[
Y \xrightarrow{i_{f}} C_{f},
\]
displayed in the diagram
\[
\xymatrix{
Y^{0} \ar[r]^{d_{Y}^{0}} \ar[d]^{\renewcommand{\arraystretch}{0.5}\begin{bmatrix}
    \small 0 \\ \small 1
\end{bmatrix}}
    & Y^{1} \ar[r]^{d_{Y}^{1}} \ar[d]^{\renewcommand{\arraystretch}{0.5}\begin{bmatrix}
        0 \\ 1
    \end{bmatrix}}
    & \cdots \ar[r]
    & Y^{n-1} \ar[r]^{d_{Y}^{n-1}} \ar[d]^{\renewcommand{\arraystretch}{0.5}\begin{bmatrix}
        0 \\ 1
    \end{bmatrix}}
    & T(Y^{0}) \ar[d]^{\renewcommand{\arraystretch}{0.5}\begin{bmatrix}
        0 \\ 1
    \end{bmatrix}}
\\
X^{1}\oplus Y^{0} \ar[r]^{d_{C_{f}}^{0}}
    & X^{2}\oplus Y^{1} \ar[r]^{\hspace{4.5mm} d_{C_{f}}^{1}}
    & \cdots \ar[r]
    & T(X^{0})\oplus Y^{n-1} \ar[r]^{d_{C_{f}}^{n-1}}
    & T(X^{1})\oplus T(Y^{0}).
}
\]
Similarly, we have the canonical projection
\[
C_{f} \xrightarrow{\pi_{f}} \Sigma X,
\]
displayed in the diagram
\[
\xymatrix{
X^{1}\oplus Y^{0} \ar[r]^{d_{C_{f}}^{0}} \ar[d]^{\renewcommand{\arraystretch}{0.5}\begin{bmatrix}
    1\ 0
\end{bmatrix}}
    & X^{2}\oplus Y^{1} \ar[r]^{\hspace{4.5mm} d_{C_{f}}^{1}}
        \ar[d]^{\renewcommand{\arraystretch}{0.5}\begin{bmatrix}
            1\ 0
        \end{bmatrix}}
    & \cdots \ar[r]
    & T(X^{0})\oplus Y^{n-1}
        \ar[r]^{d_{C_{f}}^{n-1}}
        \ar[d]^{\renewcommand{\arraystretch}{0.5}\begin{bmatrix}
            1\ 0
        \end{bmatrix}}
    & T(X^{1})\oplus T(Y^{0})
        \ar[d]^{\renewcommand{\arraystretch}{0.5}\begin{bmatrix}
            1\ 0
        \end{bmatrix}}
\\
X^{1} \ar[r]^{-d_{X}^{1}}
    & X^{2} \ar[r]^{-d_{X}^{2}}
    & \cdots \ar[r]
    & T(X^{0}) \ar[r]^{-T(d_{X}^{0})}
    & T(X^{1}).
}
\]

Next, we recall the definition of a right triangulated category as given in \cite{BM, ABM}.

Let $\mathcal{A}$ be an additive category and let
$\Sigma : \mathcal{A} \rightarrow \mathcal{A}$ be an additive  functor.
A \emph{sixtuple} in $\mathcal{A}$ consists of three objects
$X$, $Y$, and $Z$, together with morphisms
$f : X \rightarrow Y$,
$g : Y \rightarrow Z$, and
$h : Z \rightarrow \Sigma X$.
Such a sixtuple is denoted by $(X,Y,Z,f,g,h)$, or equivalently by the sequence
\[
X \stackrel{f}{\longrightarrow}
Y \stackrel{g}{\longrightarrow}
Z \stackrel{h}{\longrightarrow}
\Sigma X.
\]

A morphism $(\alpha, \beta, \gamma)$ between two sixtuples in $\mathcal{A}$
is a commutative diagram
\[
\xymatrix{
X_{1} \ar[r]^{f_{1}} \ar[d]^{\alpha}
    & Y_{1} \ar[r]^{g_{1}} \ar[d]^{\beta}
    & Z_{1} \ar[r]^{h_{1}} \ar[d]^{\gamma}
    & \Sigma X_{1} \ar[d]^{\Sigma \alpha} \\
X_{2} \ar[r]^{f_{2}}
    & Y_{2} \ar[r]^{g_{2}}
    & Z_{2} \ar[r]^{h_{2}}
    & \Sigma X_{2}.
}
\]
If in
addition $\alpha,\beta$ and $\gamma$ are isomorphisms in $\cal C$,
the morphism $(\alpha,\beta,\gamma)$ is then called an
\emph{isomorphism of sextuples.}

\begin{definition}\cite[Definition 2.1 ]{BM} and \cite[Definition 2.1]{ABM}
~A right triangulated category consists of an additive category $\mathcal{A}$
together with an additive functor $\Sigma : \mathcal{A} \rightarrow \mathcal{A}$
and a class $\nabla$ of sixtuples
\[
X \stackrel{f}{\longrightarrow}
Y \stackrel{g}{\longrightarrow}
Z \stackrel{h}{\longrightarrow}
\Sigma X
\]
satisfying the following axioms:

\begin{itemize}[leftmargin=4em]

\item[(RTR1)]
\begin{itemize}
\item[(a)]
Any sixtuple that is isomorphic to a sixtuple in $\nabla$ also belongs to $\nabla$.

\item[(b)]
Every morphism $f : X \rightarrow Y$ appears in a sixtuple in $\nabla$;
that is, there exist morphisms $g : Y \rightarrow Z$ and $h : Z \rightarrow \Sigma X$
such that
$
X \stackrel{f}{\longrightarrow} Y \stackrel{g}{\longrightarrow}
Z \stackrel{h}{\longrightarrow} \Sigma X
\in \nabla.
$

\item[(c)]
For any object $X$ in $\mathcal{A}$, the sixtuple
$
X \stackrel{1_{X}}{\longrightarrow} X
\longrightarrow 0
\longrightarrow \Sigma X
$
lies in $\nabla$.
\end{itemize}

\item[(RTR2)]
If
\[
X \stackrel{f}{\longrightarrow}
Y \stackrel{g}{\longrightarrow}
Z \stackrel{h}{\longrightarrow}
\Sigma X
\in \nabla,
\]
then the rotated sixtuple
\[
Y \stackrel{g}{\longrightarrow}
Z \stackrel{h}{\longrightarrow}
\Sigma X \stackrel{-\Sigma f}{\longrightarrow}
\Sigma Y
\]
also belongs to $\nabla$.

\item[(RTR3)]
If
\[
X_{1} \stackrel{f_{1}}{\longrightarrow} Y_{1}
\stackrel{g_{1}}{\longrightarrow} Z_{1}
\stackrel{h_{1}}{\longrightarrow} \Sigma X_{1}~~\mbox{and}~~
X_{2} \stackrel{f_{2}}{\longrightarrow} Y_{2}
\stackrel{g_{2}}{\longrightarrow} Z_{2}
\stackrel{h_{2}}{\longrightarrow} \Sigma X_{2}
\]
are in $\nabla$, and morphisms
$\alpha\colon X_{1} \rightarrow X_{2}$,
$\beta\colon Y_{1} \rightarrow Y_{2}$
make the left square commute,
then there exists a morphism $\gamma$ such that the entire diagram commutes:
\[
\xymatrix{
X_{1} \ar[r]^{f_{1}} \ar[d]^{\alpha}
    & Y_{1} \ar[r]^{g_{1}} \ar[d]^{\beta}
    & Z_{1} \ar[r]^{h_{1}} \ar@{.>}[d]^{\gamma}
    & \Sigma X_{1} \ar[d]^{\Sigma\alpha} \\
X_{2} \ar[r]^{f_{2}}
    & Y_{2} \ar[r]^{g_{2}}
    & Z_{2} \ar[r]^{h_{2}}
    & \Sigma X_{2}.
}
\]

\item[(RTR4)] {\rm \bf (Octahedral axiom)}
Given a commutative diagram
\[
\xymatrix{
X_{1} \ar[r]^{f_{1}} \ar@{=}[d]
    & Y_{1} \ar[r]^{g_{1}} \ar[d]^{\beta_{1}}
    & Z_{1} \ar[r]^{h_{1}} \ar@{.>}[d]^{\gamma_{1}}
    & \Sigma X_{1} \ar@{=}[d] \\
X_{1} \ar[r]^{\beta_{1} f_{1}}
    & Y_{2} \ar[r]^{g_{2}} \ar[d]^{\beta_{2}}
    & Z_{2} \ar[r]^{h_{2}} \ar@{.>}[d]^{\gamma_{2}}
    & \Sigma X_{1} \ar[d]^{\Sigma f_{1}} \\
& Y_{3} \ar@{=}[r] \ar[d]^{\beta_{3}}
    & Y_{3} \ar[r]^{\beta_{3}} \ar[d]^{\Sigma(g_{1}) \beta_{3}}
    & \Sigma Y_{1} \\
& \Sigma Y_{1} \ar[r]^{\Sigma g_{1}}
    & \Sigma Z_{1},
}
\]
where the first two rows and the second column are in $\nabla$,
there exist morphisms
\(\gamma_{1}\colon Z_{1} \rightarrow Z_{2}\)
and
\(\gamma_{2} \colon Z_{2} \rightarrow Y_{3}\)
such that the entire diagram commutes
and the third column is also in $\nabla$.
\end{itemize}
\end{definition}

The elements of $\nabla$ are called \emph{right triangles}.
Note that a right triangulated category is a triangulated category
if and only if $\Sigma$ is an automorphism.
The functor $\Sigma$ is called the \emph{suspension functor} of the right triangulated category $(\A,\Sigma,\nabla)$.
\vspace{2mm}

Now we define sixtuples in the homotopy category
${\rm HFact}_{n}(\mathcal{A},T,\omega)$.

\begin{definition}
A \emph{sixtuple} in
${\rm HFact}_{n}(\mathcal{A},T,\omega)$
consists of three objects
$X$, $Y$, and $Z$, together with homotopy classes of morphisms
$[f] : X \rightarrow Y$,
$[g] : Y \rightarrow Z$,
$[h] : Z \rightarrow \Sigma X$.
Such a sixtuple is written
$(X, Y, Z, [f], [g], [h])$,
or equivalently as the sequence
$$
X \stackrel{[f]}{\longrightarrow}
Y \stackrel{[g]}{\longrightarrow}
Z \stackrel{[h]}{\longrightarrow}
\Sigma X.
$$
A morphism
$([\alpha], [\beta], [\gamma])$
between two sixtuples in
${\rm HFact}_{n}(\mathcal{A},T,\omega)$
is a commutative diagram
$$
\xymatrix{
X_{1} \ar[r]^{[f_{1}]} \ar[d]^{[\alpha]}
    & Y_{1} \ar[r]^{[g_{1}]} \ar[d]^{[\beta]}
    & Z_{1} \ar[r]^{[h_{1}]} \ar[d]^{[\gamma]}
    & \Sigma X_{1} \ar[d]^{\Sigma [\alpha]} \\
X_{2} \ar[r]^{[f_{2}]}
    & Y_{2} \ar[r]^{[g_{2}]}
    & Z_{2} \ar[r]^{[h_{2}]}
    & \Sigma X_{2}.
}
$$
\end{definition}

\begin{remark}
\begin{itemize}

    \item[(1)]
    A morphism between two sixtuples is an isomorphism
    if and only if each of the vertical morphisms is an isomorphism.

    \item[(2)]
    If $f_{1}$ is homotopic to $f_{2}$ via a homotopy
    $s = (s^{0}, s^{1}, \dots, s^{n-1})$ in
    ${\rm Fact}_{n}(\mathcal{A},T,\omega)$,
    then the mapping cones $C_{f_{1}}$ and $C_{f_{2}}$
    are isomorphic in
    ${\rm HFact}_{n}(\mathcal{A},T,\omega)$.
    An isomorphism is given by
    $$
    \lambda =
    \left(
    \begin{bmatrix} 1 & 0 \\ s^{0} & 0 \end{bmatrix},
    \begin{bmatrix} 1 & 0 \\ s^{1} & 0 \end{bmatrix},
    \dots,
    \begin{bmatrix} 1 & 0 \\ s^{\,n-1} & 0 \end{bmatrix}
    \right).
    $$
    Its inverse is
    $$
    \lambda^{-1} =
    \left(
    \begin{bmatrix} 1 & 0 \\ -s^{0} & 0 \end{bmatrix},
    \begin{bmatrix} 1 & 0 \\ -s^{1} & 0 \end{bmatrix},
    \dots,
    \begin{bmatrix} 1 & 0 \\ -s^{\,n-1} & 0 \end{bmatrix}
    \right).
    $$

\end{itemize}
\end{remark}

We have now equipped ${\rm HFact}_{n}(\mathcal{A},T,\omega)$ with all the essential
homological constructions: the suspension functor, the notion of homotopy, and the mapping
cone associated to each morphism.
Moreover, we have shown that mapping cones are invariant under homotopy,
so that the canonical sixtuples
$$
X \stackrel{[f]}{\longrightarrow}
Y \stackrel{[i_{f}]}{\longrightarrow}
C_{f} \stackrel{[\pi_{f}]}{\longrightarrow}
\Sigma X
$$
are well defined in the homotopy category.
These sixtuples behave analogously to the distinguished triangles in classical homological
contexts, and thus naturally suggest a candidate for the class of right triangles in
${\rm HFact}_{n}(\mathcal{A},T,\omega)$.
The following theorem confirms that they indeed endow the homotopy category with the structure
of a right triangulated category.

\begin{theorem} \label{the3.8}
Let $\nabla$ be the class of all sixtuples that are isomorphic to one of the form
\[
X \stackrel{[f]}{\longrightarrow}
Y \stackrel{[i_{f}]}{\longrightarrow}
C_{f} \stackrel{[\pi_{f}]}{\longrightarrow}
\Sigma X .
\]
Then the homotopy category
${\rm HFact}_{n}(\mathcal{A},T,\omega)$
of $n$-fold $(\mathcal{A},T)$-factorizations of $\omega$,
equipped with the suspension  functor $\Sigma$
and the class $\nabla$,
is a right triangulated category.
\end{theorem}

\begin{proof}
We will check the axioms of right triangulated categories.

	{\bf (RTR1)} (a) By the definition of $\nabla$, it's easy to see that $\nabla$ is closed under isomorphisms in ${\rm HFact}_{n}(\mathcal{A},T,\omega)$.
	
	(b) Let $X,Y$ be in ${\rm Fact}_{n}(\mathcal{A},T,\omega)$ and $f=(f^{0},f^{1},\dots,f^{n-1}):X \rightarrow Y$ a morphism. That is, there is a commutative diagram:
$$\xymatrix{X^{0} \ar[r]^{d_{X}^{0}}\ar[d]^{f^{0}}&X^{1} \ar[r]^{d_{X}^{1}}\ar[d]^{f^{1}} &X^{2}
			\ar[r]\ar[d]^{f^{2}} &\cdots
			\ar[r]^{d_{X}^{n-2}} &X^{n-1} \ar[r]^{d_{X}^{n-1}}\ar[d]^{f^{n-1}} &T(X^{0}) \ar[d]^{T(f^{0})}\\
			Y^{0} \ar[r]^{d_{Y}^{0}}&Y^{1} \ar[r]^{d_{Y}^{1}} &Y^{2} \ar[r] &\cdots
			\ar[r]^{d_{Y}^{n-2}} &Y^{n-1} \ar[r]^{d_{Y}^{n-1}} &T(Y^{0}) . }$$

	Obviously, we can construct the mapping cone $C_{f}=\Sigma X \oplus Y$:
		\begin{center}
		$X^{1}\oplus Y^{0} \xrightarrow{d_{C_{f}}^{0}}
		X^{2}\oplus Y^{1} \xrightarrow{d_{C_{f}}^{1}} X^{3}\oplus Y^{2} \longrightarrow \cdots
		\longrightarrow T(X^{0})\oplus Y^{n-1} \xrightarrow{d_{C_{f}}^{n-1}} T(X^{1})\oplus T(Y^{0})$,
	\end{center}
	where $d_{C_{f}}^{j}=\begin{bmatrix}
		-d_{X}^{j+1} &0 \\
		f^{j+1}& d_{Y}^{j}
	\end{bmatrix}$ for any $0\leq j \leq n-2$ and $d_{C_{f}}^{n-1}=\begin{bmatrix}
		-T(d_{X}^{0}) &0 \\
		T(f^{0})& d_{Y}^{n-1}
	\end{bmatrix}$.
	Then we can get the sixtuple naturally
	$X \stackrel{[f]}{\longrightarrow} Y \stackrel{[i_{f}]}{\longrightarrow} C_{f} \stackrel{[\pi_{f}]}{\longrightarrow} \Sigma X$.

	(c) Let $X$ be an object of ${\rm Fact}_{n}(\mathcal{A},T,\omega)$.
In order to verify axiom {\rm (RTR1)(c)}, we consider the sixtuple
$$
X \xrightarrow{[\mathrm{Id}_{X}]}
X \xrightarrow{[i_{\mathrm{Id}_{X}}]}
C_{\mathrm{Id}_{X}} \xrightarrow{[\pi_{\mathrm{Id}_{X}}]}
\Sigma X
$$
in $\nabla$.  It remains to show that
$C_{\mathrm{Id}_{X}}$ is homotopic to the zero object of
${\rm Fact}_{n}(\mathcal{A},T,\omega)$.

Consider the diagram
$$
\xymatrix{
X^{1}\oplus X^{0} \ar[r]^{d_{C_{\mathrm{Id}_{X}}}^{0}}\ar[d]^{1}
    & X^{2}\oplus X^{1} \ar@{.>}[ld]_{s^{0}}
        \ar[r]^{d_{C_{\mathrm{Id}_{X}}}^{1}}\ar[d]^{1}
    & X^{3}\oplus X^{2} \ar@{.>}[ld]_{s^{1}}
        \ar[r]\ar[d]^{1}
    & \cdots \ar[r]
    & T(X^{0})\oplus X^{n-1} \ar@{.>}[ld]_{s^{\,n-2}}
        \ar[r]^{d_{C_{\mathrm{Id}_{X}}}^{\,n-1}}\ar[d]^{1}
    & T(X^{1})\oplus T(X^{0}) \ar@{.>}[ld]_{s^{\,n-1}}\ar[d]^{1} \\
X^{1}\oplus X^{0} \ar[r]^{d_{C_{\mathrm{Id}_{X}}}^{0}}
    & X^{2}\oplus X^{1} \ar[r]^{d_{C_{\mathrm{Id}_{X}}}^{1}}
    & X^{3}\oplus X^{2} \ar[r]
    & \cdots \ar[r]
    & T(X^{0})\oplus X^{n-1} \ar[r]^{d_{C_{\mathrm{Id}_{X}}}^{\,n-1}}
    & T(X^{1})\oplus T(X^{0}) .
}
$$

Define
$$
s=(s^{0},s^{1},\dots,s^{n-1}), \qquad
s^{j} =
\begin{bmatrix}
0 & 1 \\[2pt]
0 & 0
\end{bmatrix}
\quad (0 \le j \le n-1).
$$
A direct computation shows that
$$
d_{C_{\mathrm{Id}_{X}}}^{\,n-1} \circ s^{\,n-1}
    + T(s^{0}\circ d_{C_{\mathrm{Id}_{X}}}^{0})
    =
\begin{bmatrix}
1 & 0 \\[2pt]
0 & 1
\end{bmatrix},
$$
and for every $1 \le j \le n-1$,
$$
d_{C_{\mathrm{Id}_{X}}}^{\,j-1} \circ s^{\,j-1}
    + s^{j}\circ d_{C_{\mathrm{Id}_{X}}}^{j}
    =
\begin{bmatrix}
1 & 0 \\[2pt]
0 & 1
\end{bmatrix}.
$$
From the identities established above,
it follows that $C_{\mathrm{Id}_{X}}$ is homotopic to the zero object.
Consequently, the sixtuple
$$
X \xrightarrow{[\mathrm{Id}_{X}]} X \longrightarrow 0 \longrightarrow \Sigma X
$$
belongs to $\nabla$.
\vspace{2mm}

{\bf (RTR2)}
Let
$$
X \stackrel{[f]}{\longrightarrow}
Y \stackrel{[i_{f}]}{\longrightarrow}
C_{f} \stackrel{[\pi_{f}]}{\longrightarrow}
\Sigma X
$$
be a sixtuple in $\nabla$.
To verify axiom {\rm (RTR2)}, we must show that the rotated sixtuple
$$
Y \stackrel{[i_{f}]}{\longrightarrow}
C_{f} \stackrel{[\pi_{f}]}{\longrightarrow}
\Sigma X \stackrel{-\Sigma [f]}{\longrightarrow} \Sigma Y
$$
also belongs to $\nabla$.
Equivalently, we need to prove that it is homotopic to the canonical cone sixtuple
$$
Y \stackrel{[i_{f}]}{\longrightarrow}
C_{f} \stackrel{[i_{i_{f}}]}{\longrightarrow}
C_{i_{f}} \stackrel{[\pi_{i_{f}}]}{\longrightarrow}
\Sigma Y .
$$

To this end, we proceed in three steps.

\begin{itemize}
    \item[(i)]
    Define a morphism
    $
    \alpha\colon \Sigma X \longrightarrow C_{i_{f}}
    $
    in ${\rm Fact}_{n}(\mathcal{A},T,\omega)$.

    \item[(ii)]
    Show that $\alpha$ becomes an isomorphism in the homotopy category
    ${\rm HFact}_{n}(\mathcal{A},T,\omega)$.

    \item[(iii)]
    Verify that $\alpha$ fits into a morphism of sixtuples, producing a commutative diagram
    $$
    \xymatrix{
    Y \ar[r]^{[i_{f}]} \ar@{=}[d]
        & C_{f} \ar[r]^{[\pi_{f}]} \ar@{=}[d]
        & \Sigma X \ar[r]^{-\,\Sigma[f]} \ar[d]^{[\alpha]}
        & \Sigma Y \ar@{=}[d] \\
    Y \ar[r]^{[i_{f}]}
        & C_{f} \ar[r]^{[i_{i_{f}}]}
        & C_{i_{f}} \ar[r]^{[\pi_{i_{f}}]}
        & \Sigma Y ,
    }
    $$
    thereby demonstrating that the rotated sixtuple lies in $\nabla$.
\end{itemize}

	(i) Define $\alpha=(\alpha^{0},\alpha^{1},\dots,\alpha^{n-1})=(\begin{bmatrix}
		-f^{1}\\ 1 \\ 0
	\end{bmatrix},\begin{bmatrix}
		-f^{2}\\ 1 \\ 0
	\end{bmatrix},\dots,\begin{bmatrix}
		-T(f^{0})\\ 1 \\ 0
	\end{bmatrix})\colon\Sigma X \rightarrow C_{i_{f}}$. Consider the following diagram:
$$\small \xymatrix{X^{1}\ar[r]^{-d_{X}^{1}} \ar[d]^{\alpha^{0}} &X^{2} \ar[r]^{-d_{X}^{2}}\ar[d]^{\alpha^{1}} &\cdots \ar[r]&T(X^{0})\ar[r]^{-T(d_{X}^{0})} \ar[d]^{\alpha^{n-1}} &T(X^{1}) \ar[d]^{T(\alpha^{0})}\\
			Y^{1}\oplus X^{1} \oplus Y^{0}\ar[r]^{d_{C_{i_{f}}}^{0}} &Y^{2}\oplus X^{2} \oplus Y^{1}\ar[r]^(.65){d_{C_{i_{f}}}^{1}} &\cdots \ar[r]&T(Y^{0})\oplus T(X^{0}) \oplus Y^{n-1}\ar[r]^{d_{C_{i_{f}}}^{n-1}}&T(Y^{1})\oplus T(X^{1}) \oplus T(Y^{0}), }$$
where
\[
d_{C_{i_{f}}}^{j}=\begin{bmatrix}
		-d_{Y}^{j+1} &0 &0\\
		0 &-d_{X}^{j+1} &0\\
		1 &f^{j+1} &d_{Y}^{j}
	\end{bmatrix}
\quad (0 \le j \le n-2),
\qquad
d_{C_{i_{f}}}^{n-1}=\begin{bmatrix}
		-T(d_{Y}^{0}) &0 &0\\
		0 &-T(d_{X}^{0}) &0\\
		1 &T(f^{0}) &d_{Y}^{n-1}
	\end{bmatrix}.
\]

It is easy to verify that each square is commutative. Therefore, the definition is well defined.

  (ii) Define $\beta=(\beta^{0},\beta^{1},\dots,\beta^{n-1})\colon C_{i_{f}} \rightarrow \Sigma X$, where $\beta^{j}=[0\ 1\ 0]$ for any $0\leq j \leq n-1$. Consider the following diagram:
$$\footnotesize \xymatrix{Y^{1}\oplus X^{1} \oplus Y^{0}\ar[r]^{d_{C_{i_{f}}}^{0}} \ar[d]^{\beta^{0}} &Y^{2}\oplus X^{2} \oplus Y^{1} \ar[r]^(.65){d_{C_{i_{f}}}^{1}}\ar[d]^{\beta^{1}} &\cdots \ar[r]&T(Y^{0})\oplus T(X^{0}) \oplus Y^{n-1}\ar[r]^{d_{C_{i_{f}}}^{n-1}} \ar[d]^{\beta^{n-1}} &T(Y^{1})\oplus T(X^{1}) \oplus T(Y^{0}) \ar[d]^{T(\beta^{0})}\\
			X^{1}\ar[r]^{-d_{X}^{1}} &X^{2}\ar[r]^{-d_{X}^{2}} &\cdots \ar[r]&T(X^{0})\ar[r]^{-T(d_{X}^{0})}&T(X^{1}).}$$	
It is easy to verify that each square is commutative.
Hence the definition of $\beta$ is well defined.
Obviously, $\beta\alpha=1_{\Sigma X}$. Next, we shall verify
$\alpha\beta \sim 1_{C_{i_{f}}}$. Let $$h=1_{C_{i_{f}}}-\alpha\beta=(h^{0},h^{1},\dots,h^{n-1})\colon C_{i_{f}} \rightarrow C_{i_{f}},$$ where $h^{j}=\begin{bmatrix}
		1 &f^{j+1} &0\\
		0 &0 &0\\
		0 &0 &1
	\end{bmatrix}$ for any $0 \leq j \leq n-2$ and $h^{n-1}=\begin{bmatrix}
		1 &T(f^{0}) &0\\
		0 &0 &0\\
		0 &0 &1
	\end{bmatrix}$. Consider the following diagram
$$
\footnotesize
\xymatrix{
Y^{1}\oplus X^{1}\oplus Y^{0}
    \ar[r]^{d_{C_{i_{f}}}^{0}} \ar[d]^{h^{0}}
    & Y^{2}\oplus X^{2}\oplus Y^{1}
        \ar[r]^(.65){d_{C_{i_{f}}}^{1}}\ar[d]^{h^{1}}
        \ar@{.>}[ld]_{s^{0}}
    & \cdots \ar[r]
    & T(Y^{0})\oplus T(X^{0})\oplus Y^{n-1}
        \ar[r]^{d_{C_{i_{f}}}^{n-1}}\ar[d]^{h^{\,n-1}}
        \ar@{.>}[ld]_{s^{\,n-2}}
    & T(Y^{1})\oplus T(X^{1})\oplus T(Y^{0})
        \ar[d]^{T(h^{0})}
        \ar@{.>}[ld]_{s^{\,n-1}} \\
Y^{1}\oplus X^{1}\oplus Y^{0}
    \ar[r]^{d_{C_{i_{f}}}^{0}}
    & Y^{2}\oplus X^{2}\oplus Y^{1}
        \ar[r]^(.65){d_{C_{i_{f}}}^{1}}
    & \cdots \ar[r]
    & T(Y^{0})\oplus T(X^{0})\oplus Y^{n-1}
        \ar[r]^{d_{C_{i_{f}}}^{n-1}}
    & T(Y^{1})\oplus T(X^{1})\oplus T(Y^{0}).
}
$$
Define
$$
s = (s^{0}, s^{1}, \dots, s^{n-1}),
~~
s^{j} =
\begin{bmatrix}
0 & 0 & 1 \\[2pt]
0 & 0 & 0 \\[2pt]
0 & 0 & 0
\end{bmatrix}
\quad (0 \le j \le n-1).
$$
A direct computation gives
$$
d_{C_{i_{f}}}^{\,n-1} \circ s^{\,n-1}
    + T(s^{0} \circ d_{C_{i_{f}}}^{0})
    = T(h^{0}),
$$
and for every $1 \le j \le n-1$,
$$
d_{C_{i_{f}}}^{\,j-1} \circ s^{\,j-1}
    + s^{j} \circ d_{C_{i_{f}}}^{j}
    = h^{j}.
$$
From these identities it follows that $\alpha$ is an isomorphism in
${\rm HFact}_{n}(\mathcal{A},T,\omega)$,
and $\beta$ serves as its inverse.

	(\romannumeral3) Consider the following diagram
$$\xymatrix{Y\ar[r]^{\text{[$i_{f}$]}} \ar@{=}[d]&C_{f} \ar[r]^{\text{[$\pi_{f}$]}} \ar@{=}[d] &\Sigma X \ar[r]^{-\Sigma \text{[$f$]}} \ar@<0.7ex>@{.>}[d]^{\text{[$\alpha$]}} &\Sigma Y \ar@{=}[d]\\
Y\ar[r]^{\text{[$i_{f}$]}} &C_{f} \ar[r]^{\text{[$i_{i_{f}}$]}} &C_{i_{f}} \ar[r]^{\text{[$\pi_{i_{f}}$]}} \ar@<0.7ex>@{.>}[u]^{\text{[$\beta$]}} &\Sigma Y.}$$
It's easy to verify that $\pi_{i_{f}}\circ \alpha=-\Sigma f$. Now, we just need to verify that $i_{i_{f}} \sim \alpha \circ \pi _{f}$.
	Let $$h=i_{i_{f}}-\alpha \circ \pi _{f}=(h^{0},h^{1},\dots,h^{n-1}):C_{f} \rightarrow C_{i_{f}},$$ where $h^{j}=\begin{bmatrix}
		f^{j+1} &0 \\
		0 &0 \\
		0 &1
	\end{bmatrix}$ for any $0 \leq j \leq n-2$ and $h^{n-1}=\begin{bmatrix}
		T(f^{0}) &0 \\
		0 &0 \\
		0 &1
	\end{bmatrix}$. For the following diagram
$$\footnotesize \xymatrix{X^{1} \oplus Y^{0}\ar[r]^{d_{C_{f}}^{0}} \ar[d]^{h^{0}} &X^{2} \oplus Y^{1} \ar[r]^(.65){d_{C_{f}}^{1}}\ar[d]^{h^{1}} \ar@{.>}[ld]_{s^{0}} &\cdots \ar[r]&T(X^{0}) \oplus Y^{n-1}\ar[r]^{d_{C_{f}}^{n-1}} \ar[d]^{h^{n-1}} \ar@{.>}[ld]_{s^{n-2}}&T(X^{1}) \oplus T(Y^{0}) \ar[d]^{T(h^{0})} \ar@{.>}[ld]_{s^{n-1}}\\
	Y^{1}\oplus X^{1}\oplus Y^{0}\ar[r]^{d_{C_{i_{f}}}^{0}} &Y^{2}\oplus X^{2} \oplus Y^{1}\ar[r]^(.65){d_{C_{i_{f}}}^{1}} &\cdots \ar[r]&T(Y^{0})\oplus T(X^{0}) \oplus Y^{n-1}\ar[r]^{d_{C_{i_{f}}}^{n-1}}&T(Y^{1})\oplus T(X^{1}) \oplus T(Y^{0})  ,}$$
where 
$$d_{C_{f}}^{j}=\begin{bmatrix}
		-d_{X}^{j+1} &0\\
		f^{j+1} &d_{Y}^{j}
	\end{bmatrix},~~ d_{C_{i_{f}}}^{j}=\begin{bmatrix}
		-d_{Y}^{j+1} &0 &0\\
		0 &-d_{X}^{j+1} &0\\
		1 &f^{j+1} &d_{Y}^{j}
	\end{bmatrix} \ (0\leq j \leq n-2),$$
$$d_{C_{f}}^{n-1}=\begin{bmatrix}
		-T(d_{X}^{0}) &0\\
		T(f^{0}) &d_{Y}^{n-1}
	\end{bmatrix},~~d_{C_{i_{f}}}^{n-1}=\begin{bmatrix}
		-T(d_{Y}^{0}) &0 &0\\
		0 &-T(d_{X}^{0}) &0\\
		1 &T(f^{0}) &d_{Y}^{n-1}
	\end{bmatrix}.$$
	
	Let $s=(s^{0},s^{1},\dots,s^{n-1})$ with $s^{j}=\begin{bmatrix}
		0 &1\\
		0 &0\\
		0 &0
	\end{bmatrix}$ for any $0 \leq j \leq n-1$. We can verify that
	$$d_{C_{i_{f}}}^{n-1}\circ s^{n-1}+T(s^{0}\circ d_{C_{f}}^{0})=T(h^{0})~~\mbox{and}~~
	d_{C_{i_{f}}}^{j-1}\circ s^{j-1}+s^{j}\circ d_{C_{f}}^{j}=h^{j}~~\mbox{for}~~1\leq j \leq n-1.$$
	Therefore, the morphism $\alpha$ induces the morphism
$([1_{Y}],[1_{C_{f}}],[\alpha])$ of sixtuples.

	{\bf (RTR3)} Let $\alpha\colon X\rightarrow X_{1},\beta\colon Y\rightarrow Y_{1}$ be morphisms in ${\rm Fact}_{n}(\mathcal{A},T,\omega)$. Then we obtain two commutative diagrams
$$\xymatrix{X^{0} \ar[r]^{d_{X}^{0}}\ar[d]^{\alpha^{0}}&X^{1} \ar[r]^{d_{X}^{1}}\ar[d]^{\alpha^{1}} &X^{2}
		\ar[r]\ar[d]^{\alpha^{2}} &\cdots
		\ar[r]^{d_{X}^{n-2}} &X^{n-1} \ar[r]^{d_{X}^{n-1}}\ar[d]^{\alpha^{n-1}} &T(X^{0}) \ar[d]^{T(\alpha^{0})}\\
		X_{1}^{0} \ar[r]^{d_{X_{1}}^{0}}&X_{1}^{1} \ar[r]^{d_{X_{1}}^{1}} &X_{1}^{2} \ar[r] &\cdots
		\ar[r]^{d_{X_{1}}^{n-2}} &X_{1}^{n-1} \ar[r]^{d_{X_{1}}^{n-1}} &T(X_{1}^{0}),}$$
$$\xymatrix{Y^{0} \ar[r]^{d_{Y}^{0}}\ar[d]^{\beta^{0}}&Y^{1} \ar[r]^{d_{Y}^{1}}\ar[d]^{\beta^{1}} &Y^{2}
		\ar[r]\ar[d]^{\beta^{2}} &\cdots
		\ar[r]^{d_{Y}^{n-2}} &Y^{n-1} \ar[r]^{d_{Y}^{n-1}}\ar[d]^{\beta^{n-1}} &T(Y^{0}) \ar[d]^{T(\beta^{0})}\\
		Y_{1}^{0} \ar[r]^{d_{Y_{1}}^{0}}&Y_{1}^{1} \ar[r]^{d_{Y_{1}}^{1}} &Y_{1}^{2} \ar[r] &\cdots
		\ar[r]^{d_{Y_{1}}^{n-2}} &Y_{1}^{n-1} \ar[r]^{d_{Y_{1}}^{n-1}} &T(Y_{1}^{0}) .}$$	
Consider the following diagram
$$\xymatrix{X\ar[r]^{\text{[$f$]}} \ar[d]^{\text{[$\alpha$]}}&Y \ar[r]^{\text{[$i_{f}$]}} \ar[d]^{\text{[$\beta$]}} &C_{f} \ar[r]^{\text{[$\pi_{f}$]}} \ar@{.>}[d]^{\text{[$\gamma$]}} &\Sigma X \ar[d]^{\Sigma \text{[$\alpha$]}}\\
	X_{1}\ar[r]^{\text{[$g$]}} &Y_{1} \ar[r]^{\text{[$i_{g}$]}} &C_{g} \ar[r]^{\text{[$\pi_{g}$]}} &\Sigma X_{1}.}$$
Assume that the left square commutes in
${\rm HFact}_{n}(\mathcal{A},T,\omega)$.
We must show that there exists a morphism
$\gamma : C_{f} \rightarrow C_{g}$
rendering the two remaining squares on the right commutative in
${\rm HFact}_{n}(\mathcal{A},T,\omega)$.
For $0 \le j \le n-1$, set
$$h^{j} = \beta^{j} f^{j} - g^{j} \alpha^{j}.
$$ By assumption, the equality $[g]\circ[\alpha] = [\beta]\circ[f]$ holds in the homotopy category, and therefore there exists a sequence
$s = (s^{0}, s^{1}, \dots, s^{n-1})$ making the diagram
$$
\xymatrix{
X^{0} \ar[r]^{d_{X}^{0}}\ar[d]_{h^{0}}
    & X^{1} \ar@{.>}[ld]_{s^{0}} \ar[r]^{d_{X}^{1}} \ar[d]_{h^{1}}
    & X^{2} \ar@{.>}[ld]_{s^{1}} \ar[r] \ar[d]_{h^{2}}
    & \cdots \ar[r]^{d_{X}^{n-2}}
    & X^{n-1} \ar@{.>}[ld]_{s^{n-2}} \ar[r]^{d_{X}^{n-1}} \ar[d]_{h^{n-1}}
    & T(X^{0}) \ar@{.>}[ld]_{s^{n-1}} \ar[d]_{T(h^{0})} \\
Y_{1}^{0} \ar[r]^{d_{Y_{1}}^{0}}
    & Y_{1}^{1} \ar[r]^{d_{Y_{1}}^{1}}
    & Y_{1}^{2} \ar[r]
    & \cdots \ar[r]^{d_{Y_{1}}^{n-2}}
    & Y_{1}^{n-1} \ar[r]^{d_{Y_{1}}^{n-1}}
    & T(Y_{1}^{0})
}
$$
commutative.
Explicitly, the homotopy relations give:
$$
d_{Y_{1}}^{\,n-1} \circ s^{\,n-1} + T(s^{0} \circ d_{X}^{0}) = T(h^{0}),
$$
and for every $1 \le j \le n-1$,
$$
d_{Y_{1}}^{\,j-1} \circ s^{\,j-1} + s^{j} \circ d_{X}^{j} = h^{j}.
$$
Define a morphism
$$
\gamma = (\gamma^{0}, \gamma^{1}, \dots, \gamma^{n-1}) : C_{f} \rightarrow C_{g},
$$
where
$$
\gamma^{j} =
\begin{bmatrix}
\alpha^{\,j+1} & 0 \\
s^{\,j}        & \beta^{\,j}
\end{bmatrix}
\quad (0 \le j \le n-2),
~~
\gamma^{\,n-1} =
\begin{bmatrix}
T(\alpha^{0}) & 0 \\
s^{\,n-1}     & \beta^{\,n-1}
\end{bmatrix}.
$$
One checks directly that these formulas are compatible with the defining
diagrams and homotopy relations above.
Moreover,
$$
\gamma \circ i_{f} = i_{g} \circ \beta,
~~
\pi_{g} \circ \gamma = \Sigma(\alpha) \circ \pi_{f},
$$
so the two squares on the right commute in
${\rm HFact}_{n}(\mathcal{A},T,\omega)$.
Thus, axiom {\rm (RTR3)} is satisfied.

{(\bf RTR4)}
Let $f\colon X \rightarrow Y$ and $g\colon Y \rightarrow Z$ be morphisms in
${\rm Fact}_{n}(\mathcal{A},T,\omega)$. Then we obtain the commutative diagram
$$
\xymatrix{
X^{0} \ar[r]^{d_{X}^{0}}\ar[d]^{f^{0}}
    &X^{1} \ar[r]^{d_{X}^{1}}\ar[d]^{f^{1}}
    &X^{2} \ar[r]\ar[d]^{f^{2}}
    &\cdots \ar[r]^{d_{X}^{n-2}}
    &X^{n-1} \ar[r]^{d_{X}^{n-1}}\ar[d]^{f^{n-1}}
    &T(X^{0}) \ar[d]^{T(f^{0})} \\
Y^{0} \ar[r]^{d_{Y}^{0}}\ar[d]^{g^{0}}
    &Y^{1} \ar[r]^{d_{Y}^{1}}\ar[d]^{g^{1}}
    &Y^{2} \ar[r]\ar[d]^{g^{2}}
    &\cdots \ar[r]^{d_{Y}^{n-2}}
    &Y^{n-1} \ar[r]^{d_{Y}^{n-1}}\ar[d]^{g^{n-1}}
    &T(Y^{0}) \ar[d]^{T(g^{0})} \\
Z^{0} \ar[r]^{d_{Z}^{0}}
    &Z^{1} \ar[r]^{d_{Z}^{1}}
    &Z^{2} \ar[r]
    &\cdots \ar[r]
    &Z^{n-1} \ar[r]^{d_{Z}^{n-1}}
    &T(Z^{0}).
}
$$

Now consider the diagram in ${\rm HFact}_{n}(\mathcal{A},T,\omega)$
$$
\xymatrix{
X \ar[r]^{[f]} \ar@{=}[d]
    &Y \ar[r]^{[i_{f}]} \ar[d]^{[g]}
    &C_{f} \ar[r]^{[\pi_{f}]} \ar@{.>}[d]^{[\alpha]}
    &\Sigma X \ar@{=}[d] \\
X \ar[r]^{[gf]}
    &Z \ar[r]^{[i_{gf}]} \ar[d]^{[i_{g}]}
    &C_{gf} \ar[r]^{[\pi_{gf}]} \ar@{.>}[d]^{[\beta]}
    &\Sigma X \ar[d]^{\Sigma [f]} \\
& C_{g} \ar@{=}[r] \ar[d]^{[\pi_{g}]}
    &C_{g} \ar[r]^{[\pi_{g}]} \ar[d]^{[\gamma]}
    &\Sigma Y \\
&\Sigma Y \ar[r]^{\Sigma[i_{f}]}
    &\Sigma C_{f}
}
$$
where the first two rows and the second column lie in $\nabla$ and $\gamma = \Sigma(i_{f}) \circ \pi_{g}$.
To obtain the third column in $\nabla$ as required by axiom (RTR4), it suffices to carry out the following three steps.
\begin{itemize}
\item[(i)] \textbf{Define $\alpha$ and $\beta$.}
Construct morphisms
$$
\alpha\colon C_{f} \to C_{gf},
\qquad
\beta\colon C_{gf} \to C_{g},
$$
componentwise so that they extend the commutativity already imposed by $f$, $g$, and $gf$.

\item[(ii)] \textbf{Verify commutativity in ${\rm HFact}_{n}(\mathcal{A},T,\omega)$.}
One checks that the equalities
$$
[\alpha]\circ[i_{f}] = [i_{gf}]\circ[g],\qquad
[\beta]\circ[i_{gf}] = [i_{g}],\qquad
[\pi_{gf}] \circ [\alpha]=[\pi_{f}],\qquad \Sigma[f]\circ[\pi_{gf}] = [\pi_{g}] \circ [\beta]
$$
all hold up to homotopy.
Thus the entire diagram above is commutative in the homotopy category.

\item[(iii)] \textbf{Show that the lower row is a sixtuple.}
With $
\gamma = \Sigma(i_{f}) \circ \pi_{g},
$ the sixtuple
$$
C_{f} \xrightarrow{[\alpha]} C_{gf}
    \xrightarrow{[\beta]} C_{g}
    \xrightarrow{[\gamma]} \Sigma C_{f}
$$
is isomorphic to the mapping-cone sixtuple associated to $gf$.
Hence it belongs to $\nabla$.
\end{itemize}

	(i) Let $\alpha=(\alpha^{0},\alpha^{1},\dots,\alpha^{n-1})=(\setlength{\arraycolsep}{3pt}\begin{bmatrix} 1&0\\0&g^{0} \end{bmatrix},\begin{bmatrix} 1&0\\0&g^{1} \end{bmatrix},\dots,\begin{bmatrix} 1&0\\0&g^{n-1} \end{bmatrix})\colon C_{f} \rightarrow C_{gf},$
$$\beta=(\beta^{0},\beta^{1},\dots,\\ \beta^{n-1})=(\begin{bmatrix} f^{1}&0\\0&1 \end{bmatrix},\begin{bmatrix} f^{2}&0\\0&1 \end{bmatrix},\dots,\begin{bmatrix} T(f^{0})&0\\0&1 \end{bmatrix})\colon C_{gf} \rightarrow C_{g}$$ be morphisms. More precisely, we can prove that the following diagram commutes
$$\xymatrix{X^{1} \oplus Y^{0} \ar[r]^{\alpha^{0}} \ar[d]^{d_{C_{f}}^{0}} &X^{1} \oplus Z^{0} \ar[r]^{\beta^{0}} \ar[d]^{d_{C_{gf}}^{0}} &Y^{1} \oplus Z^{0}  \ar[d]^{d_{C_{g}}^{0}}\\
			X^{2} \oplus Y^{1} \ar[r]^{\alpha^{1}} \ar[d]^{d_{C_{f}}^{1}} &X^{2} \oplus Z^{1} \ar[r]^{\beta^{1}} \ar[d]^{d_{C_{gf}}^{1}} &Y^{2} \oplus Z^{1}  \ar[d]^{d_{C_{g}}^{1}} \\
			X^{3} \oplus Y^{2} \ar[r]^{\alpha^{2}} \ar[d]^{d_{C_{f}}^{2}} &X^{3} \oplus Z^{2} \ar[r]^{\beta^{2}} \ar[d]^{d_{C_{gf}}^{2}} &Y^{3} \oplus Z^{2}  \ar[d]^{d_{C_{g}}^{2}} &\\
			\vdots \ar[d] & \vdots \ar[d]& \vdots \ar[d]\\
			X^{n-1} \oplus Y^{n-2} \ar[r]^{\alpha^{n-2}} \ar[d]^{d_{C_{f}}^{n-2}} &X^{n-1} \oplus Z^{n-2} \ar[r]^{\beta^{n-2}} \ar[d]^{d_{C_{gf}}^{n-2}} &Y^{n-1} \oplus Z^{n-2}  \ar[d]^{d_{C_{g}}^{n-2}} \\
			T(X^{0}) \oplus Y^{n-1} \ar[r]^{\alpha^{n-1}} \ar[d]^{d_{C_{f}}^{n-1}} &T(X^{0}) \oplus Z^{n-1} \ar[r]^{\beta^{n-1}} \ar[d]^{d_{C_{gf}}^{n-1}} &T(Y^{0}) \oplus Z^{n-1}  \ar[d]^{d_{C_{g}}^{n-1}} \\
			T(X^{1}) \oplus T(Y^{0}) \ar[r]^{T(\alpha^{0})} &T(X^{1}) \oplus T(Z^{0}) \ar[r]^{T(\beta^{0})}  &T(Y^{1}) \oplus T(Z^{0}).}$$

\hspace{-4mm}Each square in the diagram above is commutative. Hence the definitions are well defined.

	(ii) By the constructions of $\alpha,\beta$, it's easy to verify the commutativity of the original diagram.
	
	(iii) We need to show that there exists an isomorphism $\sigma:C_{g} \rightarrow C_{\alpha}$ which making the following diagram commutative
$$\xymatrix{C_{f}\ar[r]^{\text{[$\alpha$]}} \ar@{=}[d]&C_{gf} \ar[r]^{\text{[$\beta$]}} \ar@{=}[d] &C_{g} \arrow[r]^{\text{[$\gamma$]}} \ar@<0.7ex>@{.>}[d]^{\text{[$\sigma$]}} &\Sigma C_{f} \ar@{=}[d]\\
			C_{f}\ar[r]^{\text{[$\alpha$]}} &C_{gf} \ar[r]^{\text{[$i_{\alpha}$]}} &C_{\alpha} \ar[r]^{\text{[$\pi_{\alpha}$]}} \ar@<0.7ex>@{.>}[u]^{\text{[$\tau$]}}&\Sigma C_{f}  .}$$

	Define morphisms
$$\sigma=(\sigma^{0},\sigma^{1},\dots,\sigma^{n-1}):C_{g} \rightarrow C_{\alpha},$$
	$$\tau=(\tau^{0},\tau^{1},\dots,\tau^{n-1}):C_{\alpha} \rightarrow C_{g},$$
where
$$\sigma^{j}=\begin{bmatrix}
		0 &0 \\1 &0\\ 0 &0 \\0 &1
	\end{bmatrix} (0\leq j \leq n-1),\ \tau^{j}=\begin{bmatrix}
		0 &1 &f^{j+1} &0\\ 0 &0 &0 &1
	\end{bmatrix} (0\leq j \leq n-2), \ \tau^{n-1}=\begin{bmatrix}
		0 &1 &T(f^{0}) &0\\ 0 &0 &0 &1
	\end{bmatrix}.$$
It is easy to verify that the definitions are well defined. Additionally,
	it's obvious that $\tau \sigma=1_{C_{g}}$. Next, we just need to show that $\sigma \tau\sim 1_{C_{\alpha}}$. Let
	$$h=1_{C_{\alpha}}-\sigma \tau=(h^{0},h^{1},\dots,h^{n-1}):C_{\alpha} \rightarrow C_{\alpha},$$
 where
 $$h^{j}=\begin{bmatrix}
		1&0&0&0\\ 0&0&-f^{j+1}&0\\ 0&0&1&0\\ 0&0&0&0
	\end{bmatrix}\ (0 \leq j \leq n-2),\qquad h^{n-1}=\begin{bmatrix}
		1&0&0&0\\ 0&0&-T(f^{0})&0\\ 0&0&1&0\\ 0&0&0&0
	\end{bmatrix}.$$
Consider the following diagram
$$\small \xymatrix{X^{2} \oplus Y^{1}\oplus X^{1} \oplus Z^{0}\ar[r]^{h^{0}} \ar[d]_{d_{C_{\alpha}}^{0}} &X^{2} \oplus Y^{1}\oplus X^{1} \oplus Z^{0}\ar[d]^{d_{C_{\alpha}}^{0}}\\
			X^{3} \oplus Y^{2}\oplus X^{2} \oplus Z^{1}\ar[r]^{h^{1}} \ar[d]_{d_{C_{\alpha}}^{1}} \ar@{.>}[ur]^{s^{0}}&X^{3} \oplus Y^{2}\oplus X^{2} \oplus Z^{1}\ar[d]^{d_{C_{\alpha}}^{1}}\\
			X^{4} \oplus Y^{3}\oplus X^{3} \oplus Z^{2}\ar[r]^{h^{2}} \ar[d]_{d_{C_{\alpha}}^{2}} \ar@{.>}[ur]^{s^{1}}&X^{4} \oplus Y^{3}\oplus X^{3} \oplus Z^{2}\ar[d]^{d_{C_{\alpha}}^{2}}\\
			\vdots \ar[d]&\vdots \ar[d]\\
			T(X^{0})\oplus Y^{n-1}\oplus X^{n-1} \oplus Z^{n-2}\ar[r]^{h^{n-2}} \ar[d]_{d_{C_{\alpha}}^{n-2}} &T(X^{0})\oplus Y^{n-1}\oplus X^{n-1} \oplus Z^{n-2}\ar[d]_{d_{C_{\alpha}}^{n-2}}\\
			T(X^{1})\oplus T(Y^{0})\oplus T(X^{0}) \oplus Z^{n-1}\ar[r]^{h^{n-1}} \ar[d]_{d_{C_{\alpha}}^{n-1}} \ar@{.>}[ur]^{s^{n-2}}&T(X^{1})\oplus T(Y^{0})\oplus T(X^{0}) \oplus Z^{n-1}\ar[d]^{d_{C_{\alpha}}^{n-1}}\\
			T(X^{2})\oplus T(Y^{1})\oplus T(X^{1}) \oplus T(Z^{0})\ar[r]^{T(h^{0})}  \ar@{.>}[ur]^{s^{n-1}}&T(X^{2})\oplus T(Y^{1})\oplus T(X^{1}) \oplus T(Z^{0})   ,}$$
where
$$\small \hspace{20mm}d_{C_{\alpha}}^{j}=\begin{bmatrix}
		d_{X}^{j+2} &0 &0 &0\\
		-f^{j+2} &-d_{Y}^{j+1} &0 &0\\
		1 &0 &-d_{X}^{j+1} &0\\
		0 &g^{j+1} &g^{j+1}f^{j+1} &d_{Z}^{j}
	\end{bmatrix}\ (0\leq j \leq n-3),$$
	$$\small d_{C_{\alpha}}^{n-2}=\begin{bmatrix}
		T(d_{X}^{0}) &0 &0 &0\\
		-T(f^{0}) &-d_{Y}^{n-1} &0 &0\\
		1 &0 &-d_{X}^{n-1} &0\\
		0 &g^{n-1} &g^{n-1}f^{n-1} &d_{Z}^{n-2}
	\end{bmatrix},$$
    $$\small d_{C_{\alpha}}^{n-1}=\begin{bmatrix}
		T(d_{X}^{1}) &0 &0 &0\\
		-T(f^{1}) &-T(d_{Y}^{0}) &0 &0\\
		1 &0 &-T(d_{X}^{0}) &0\\
		0 &T(g^{0}) &T(g^{0}f^{0}) &d_{Z}^{n-1}
	\end{bmatrix}.$$

	Let $s=(s^{0},s^{1},\dots,s^{n-1})$ with $s^{j}=\begin{bmatrix}
		0 &0 &1 &0\\
		0 &0 &0 &0\\
		0 &0 &0 &0\\
		0 &0 &0 &0
	\end{bmatrix}$ for any $0 \leq j \leq n-1$. We can verify
	$d_{C_{\alpha}}^{n-1}\circ s^{n-1}+T(s^{0}\circ d_{C_{\alpha}}^{0})=T(h^{0})$ and
	$d_{C_{\alpha}}^{j-1}\circ s^{j-1}+s^{j}\circ d_{C_{\alpha}}^{j}=h^{j}$ for $1\leq j \leq n-1$.
	This shows that $\sigma \tau \sim 1_{C_{\alpha}}$, $\sigma$ is an isomorphism in ${\rm HFact}_{n}(\mathcal{A},T,\omega)$ and it's inverse is $\tau$.

	Finally, we need to prove that $\sigma$ induces the morphism between sixtuples. It's easy to see that $\pi_{\alpha}\circ \sigma=\gamma$. So we just need to prove $i_{\alpha}\sim \sigma \beta$. Let
	$$h=i_{\alpha}- \sigma \beta=(h^{0},h^{1},\dots,h^{n-1}):C_{gf} \rightarrow C_{\alpha},$$
 where
$$\small h^{j}=\begin{bmatrix}
		0&0\\ -f^{j+1}&0\\ 1&0\\ 0&0
	\end{bmatrix}\ (0 \leq j \leq n-2),\qquad h^{n-1}=\begin{bmatrix}
		0&0\\ -T(f^{0})&0\\ 1&0\\ 0&0
	\end{bmatrix}.$$
Consider the following diagram
$$\small \xymatrix{X^{1} \oplus Z^{0}\ar[r]^{h^{0}} \ar[d]_{d_{C_{gf}}^{0}} &X^{2} \oplus Y^{1}\oplus X^{1} \oplus Z^{0}\ar[d]^{d_{C_{\alpha}}^{0}}\\
			X^{2} \oplus Z^{1}\ar[r]^{h^{1}} \ar[d]_{d_{C_{gf}}^{1}} \ar@{.>}[ur]{s^{0}}&X^{3} \oplus Y^{2}\oplus X^{2} \oplus Z^{1}\ar[d]^{d_{C_{\alpha}}^{1}}\\
			X^{3} \oplus Z^{2}\ar[r]^{h^{2}} \ar[d]_{d_{C_{gf}}^{2}} \ar@{.>}[ur]^{s^{1}}&X^{4} \oplus Y^{3}\oplus X^{3} \oplus Z^{2}\ar[d]^{d_{C_{\alpha}}^{2}}\\
			\vdots \ar[d]&\vdots \ar[d]\\
			X^{n-1} \oplus Z^{n-2}\ar[r]^(.35){h^{n-2}} \ar[d]_{d_{C_{gf}}^{n-2}} &T(X^{0})\oplus Y^{n-1}\oplus X^{n-1} \oplus Z^{n-2}\ar[d]^{d_{C_{\alpha}}^{n-2}}\\
			T(X^{0}) \oplus Z^{n-1}\ar[r]^(.35){h^{n-1}} \ar[d]_{d_{C_{gf}}^{n-1}} \ar@{.>}[ur]^{s^{n-2}}&T(X^{1})\oplus T(Y^{0})\oplus T(X^{0}) \oplus Z^{n-1}\ar[d]^{d_{C_{\alpha}}^{n-1}}\\
			T(X^{1}) \oplus T(Z^{0})\ar[r]^(.35){T(h^{0})}  \ar@{.>}[ur]^{s^{n-1}}&T(X^{2})\oplus T(Y^{1})\oplus T(X^{1}) \oplus T(Z^{0}).}$$
Let $s=(s^{0},s^{1},\dots,s^{n-1})$ with $s^{j}=\begin{bmatrix}
		1 &0\\
		0 &0\\
		0 &0\\
		0 &0
	\end{bmatrix}$ for any $0 \leq j \leq n-1$. We can verify
	$d_{C_{\alpha}}^{n-1}\circ s^{n-1}+T(s^{0}\circ d_{C_{gf}}^{0})=T(h^{0})$ and
	$d_{C_{\alpha}}^{j-1}\circ s^{j-1}+s^{j}\circ d_{C_{gf}}^{j}=h^{j}$ for $1\leq j \leq n-1$.
	This shows that $i_{\alpha} \sim \sigma \beta$.
\end{proof}

\begin{corollary}{\rm \cite[Theorem 2.4]{BJ}}
	If  $T\colon \mathcal{A} \rightarrow \mathcal{A}$ is an automorphism, then the suspension functor $\Sigma\colon {\rm HFact}_{n}(\mathcal{A},T,\omega) \rightarrow {\rm HFact}_{n}(\mathcal{A},T,\omega)$ is an automorphism with the inverse $\Sigma^{-1}$ given by right rotation. Therefore, ${\rm HFact}_{n}(\mathcal{A},T,\omega)$ is a triangulated category.
\end{corollary}

\begin{remark}
In \cite[Theorem 2.4]{BJ}, they employ an approach via differential graded (DG) categories by constructing a DG enhancement of ${\rm HFact}_{n}(\mathcal{A},T,\omega)$.
In contrast, our method verifies directly from the definitions that the axioms of a triangulated category are satisfied. Hence the two approaches are completely different in nature.
\end{remark}

\section{The Frobenius exactness of $\text{Fact}_{n}(\mathcal{A},T,\omega)$}

Let $T\colon\mathcal{A} \rightarrow \mathcal{A}$ be an autoequivalence. Take a quasi-inverse $T^{-1}$ of $T$, which fits into an adjoint pair $(T^{-1},T)$. We have the unit $\eta:{\rm Id}_{\mathcal{A}} \rightarrow TT^{-1}$ and the counit $\epsilon:T^{-1}T \rightarrow {\rm Id}_{\mathcal{A}}$. Furthermore, we have the following triangular identities of adjoint functors.
\begin{equation} \label{1}
	T\epsilon \circ \eta_{T}=\text{Id}_{T}.
\end{equation}
\begin{equation} \label{2}
	\epsilon_{T^{-1}} \circ T^{-1}\eta={\rm Id}_{T^{-1}}.
\end{equation}
Define a natural transformation $\omega^{(-1)}=\epsilon\circ T^{-1}\omega:T^{-1} \rightarrow {\rm Id}_{\mathcal{A}}$. We have the following key equalities. These can be referred to in detail in \cite[Remark 2.1]{C}.
\begin{equation} \label{3}
	T\omega^{(-1)} \circ \eta=\omega.
\end{equation}
\begin{equation} \label{4}
	\omega \circ \epsilon =\omega^{(-1)} T.
\end{equation}
\begin{equation} \label{5}
	\epsilon T^{-1} \circ T^{-1}\omega T^{-1}=T^{-1}\epsilon \circ T^{-2}\omega.
\end{equation}
\begin{equation} \label{6}
	\omega T^{-1}=\eta \circ \omega^{(-1)}.
\end{equation}
Let $C$ be an object in $\mathcal{A}$. Define
$$\theta^{0}(C)=C \stackrel{{\rm Id}_{C}}{\longrightarrow} C \stackrel{{\rm Id}_{C}}{\longrightarrow} C \stackrel{{\rm Id}_{C}}{\longrightarrow} C \rightarrow \dots  \stackrel{{\rm Id}_{C}}{\longrightarrow} C
\stackrel{\omega_{C}}{\longrightarrow} T(C);$$
$$\theta^{1}(C)=T^{-1}C \stackrel{\omega_{C}^{(-1)}}{\longrightarrow} C \stackrel{{\rm Id}_{C}}{\longrightarrow} C \stackrel{{\rm Id}_{C}}{\longrightarrow} C \rightarrow \dots  \stackrel{{\rm Id}_{C}}{\longrightarrow} C
\stackrel{\eta_{C}}{\longrightarrow} TT^{-1}C.$$
It is straightforward to verify that both $\theta^{0}(C)$ and $\theta^{1}(C)$ are indeed $n$-fold $(\mathcal{A},T)$-factorizations of $\omega$, since they satisfy the identities in \eqref{6} and \eqref{3}. More precisely, they fulfill the following equations.
{\small $$
\left\{
\begin{aligned}
\omega_{C}\circ {\rm Id}_{C} \circ {\rm Id}_{C}\circ \dots \circ {\rm Id}_{C}
&= \omega_{C},\\[4pt]
T({\rm Id}_{C}) \circ \omega_{C}\circ {\rm Id}_{C} \dots \circ {\rm Id}_{C}
&= \omega_{C},\\[4pt]
\vdots \quad & \\[-2pt]
T({\rm Id}_{C}) \circ T({\rm Id}_{C}) \circ T({\rm Id}_{C}) \circ \dots \circ \omega_{C}
&= \omega_{C},
\end{aligned}
\right.
$$
$$
\left\{
\begin{aligned}
\eta_{C}\circ {\rm Id}_{C} \circ {\rm Id}_{C}\circ \dots \circ \omega_{C}^{(-1)}
&= \omega_{T^{-1}C},\\[4pt]
T(\omega_{C}^{(-1)}) \circ \eta_{C}\circ {\rm Id}_{C} \dots \circ {\rm Id}_{C}
&= \omega_{C},\\[4pt]
\vdots \quad & \\[-2pt]
T({\rm Id}_{C}) \circ T({\rm Id}_{C}) \circ T({\rm Id}_{C}) \circ \dots \circ T(\omega_{C}^{(-1)}) \circ \eta_{C}
&= \omega_{C}.
\end{aligned}
\right.
$$}
We thus obtain two functors
$$
\theta^{s}\colon \mathcal{A} \longrightarrow {\rm Fact}_{n}(\mathcal{A},T,\omega),
~~ s = 0,1,
$$
together with the projection functors
$$
{\rm pr}^{s} : {\rm Fact}_{n}(\mathcal{A},T,\omega) \longrightarrow \mathcal{A},
~~s = 0,1,\dots ,n-1,
$$
which send an $n$-fold $(\mathcal{A},T)$-factorization $X$ to its $s$-th component $X^{s}$.

For every $n$-fold $(\mathcal{A},T)$-factorization $X$ of $\omega$, we also have the shifted $n$-fold $(\mathcal{A},T)$-factorization
$$
S(X)
  = X^{1}\xrightarrow{d_{X}^{1}}
    X^{2}\xrightarrow{d_{X}^{2}}
    X^{3}\longrightarrow \cdots
    \xrightarrow{d_{X}^{\,n-2}}
    X^{\,n-1}\xrightarrow{d_{X}^{\,n-1}}
    T(X^{0})\xrightarrow{T(d_{X}^{0})}
    T(X^{1}).
$$

\begin{definition}\cite[Definition 2.1]{BT}
Let $\mathcal{A}$ be an additive category. For a sequence $X\stackrel{f}\rightarrow Y\stackrel{g}\rightarrow Z $ in $\mathcal A$, the pair of composable morphisms $(f,g)$ is called a kernel-cokernel pair, if $f=\text{ker} g$, $g=\text{coker} f$. Let $\mathscr E$ be a class of kernel-cokernel pairs. If $(f,g) \in \mathscr{E}$, then $f$ is called an inflation; $g$ is called a deflation. The short exact sequence $X\stackrel{f}\rightarrow Y\stackrel{g}\rightarrow Z$ is called a conflation.

 An exact structure on $\mathcal{A}$ is a class $\mathcal{E}$ of kernel-cokernel pairs which is closed under isomorphisms and satisfies the following axioms:

 [E0] For each object $A \in \mathcal{A}$, the identity morphism $1_{A}$ is a deflation.

 [E0$^{\text{op}}$] For each object $A \in \mathcal{A}$, the identity morphism $1_{A}$ is an inflation.

[E1] The class of deflations is closed under composition.

[E1$^{\text{op}}$] The class of inflations is closed under composition.

[E2] The pullback of a deflation along an arbitrary morphism exists and yields a deflation.

[E2$^{\text{op}}$] The pushout of an inflation along an arbitrary morphism exists and yields an inflation.

An exact category is a pair $(\mathcal{A},\mathcal{E})$ consisting of an additive category $\mathcal{A}$ and an exact structure $\mathcal{E}$ on $\mathcal{A}$.
\end{definition}

\begin{remark}
As noted in \cite{KB} and in \cite[Remark~2.4]{BT}, the axioms \([{\rm E0}]\) and \([{\rm E0}^{\mathrm{op}}]\) may be equivalently replaced by the requirement that the identity morphism of the zero object is a deflation. We refer to this equivalent condition as \([{\rm E0}^{'}]\).
\end{remark}

\begin{lemma} \label{lem1}
If every morphism in $\mathcal{A}$ has a kernel, then every morphism in ${\rm Fact}_{n}(\mathcal{A},T,\omega)$ also has a kernel, which is given componentwise.
\end{lemma}

\begin{proof}
	Let $\varphi=(\varphi^{0},\varphi^{1},\dots,\varphi^{n-1}):X \rightarrow Y$ be a morphism in $\text{Fact}_{n}(\mathcal{A},T,\omega)$ and $k^{j}:K^{j} \rightarrow X^{j}$ a kernel of $\varphi^{j}$ for $j=0,1,\dots, n-1$. There exists a unique morphism $d_{K}^{j}:K^{j} \rightarrow K^{j+1}$ satisfying $k^{j+1}\circ d_{K}^{j}=d_{X}^{j} \circ k^{j}$ for $j=0,1,\dots, n-2$ and $T(k^{0})\circ d_{K}^{n-1}=d_{X}^{n-1}\circ k^{n-1}$ for $j=n-1$. Thus we obtain the commutative diagram: {\small
$$\xymatrix{K^{0} \ar[r]^{d_{K}^{0}}\ar[d]^{k^{0}}&K^{1} \ar[r]^{d_{K}^{1}}\ar[d]^{k^{1}} &K^{2}
			\ar[r]\ar[d]^{k^{2}} &\cdots \ar[r]^{d_{K}^{n-3}}&K^{n-2}
			\ar[r]^{d_{K}^{n-2}}\ar[d]^{k^{n-2}}  &K^{n-1} \ar[r]^{d_{K}^{n-1}}\ar[d]^{k^{n-1}} &T(K^{0}) \ar[d]^{T(k^{0})}\\
			X^{0} \ar[r]^{d_{X}^{0}}\ar[d]^{\varphi^{0}}&X^{1} \ar[r]^{d_{X}^{1}}\ar[d]^{\varphi^{1}} &X^{2}
			\ar[r]\ar[d]^{\varphi^{2}} &\cdots
			\ar[r]^{d_{X}^{n-3}} &X^{n-2}
			\ar[r]^{d_{X}^{n-2}} \ar[d]^{\varphi^{n-2}}&X^{n-1} \ar[r]^{d_{X}^{n-1}}\ar[d]^{\varphi^{n-1}} &T(X^{0}) \ar[d]^{T(\varphi^{0})}\\
			Y^{0} \ar[r]^{d_{Y}^{0}}&Y^{1} \ar[r]^{d_{Y}^{1}} &Y^{2} \ar[r] &\cdots
			\ar[r]^{d_{Y}^{n-3}} &Y^{n-2}
			\ar[r]^{d_{Y}^{n-2}} &Y^{n-1} \ar[r]^{d_{Y}^{n-1}} &T(Y_{0})  .}$$  }
		
	Claim (a): $K=(K^{j},d_{K}^{j})$ is an $n$-fold $(\mathcal{A},T)$-factorization of $\omega$.
	
	\begin{itemize}
		\item [(1)]
		$T(k^{0})\circ d_{K}^{n-1} \circ d_{K}^{n-2} \circ \dots \circ d_{K}^{1} \circ d_{K}^{0}
		=d_{X}^{n-1} \circ k^{n-1}\circ d_{K}^{n-2} \circ \dots \circ d_{K}^{1} \circ d_{K}^{0}
		=\dots
		=\omega_{X^{0}}\circ k^{0}
		=T(k^{0})\circ \omega_{K^{0}}$.
		
		Since $T(k^{0})$ is injective, we have $d_{K}^{n-1} \circ d_{K}^{n-2} \circ \dots \circ d_{K}^{1} \circ d_{K}^{0}=\omega_{K^{0}}$.
		
		\item [(2)]$T(k^{1}) \circ T(d_{K}^{0})\circ d_{K}^{n-1}\circ \dots \circ d_{K}^{2} \circ d_{K}^{1}
		=T(d_{X}^{0})\circ T(k^{0})\circ d_{K}^{n-1}\circ \dots \circ d_{K}^{2} \circ d_{K}^{1}
		=\dots
		=\omega_{X^{1}}\circ k^{1}
		=T(k^{1})\circ \omega_{K^{1}}$.
		
		Since $T(k^{1})$ is injective, we have $T(d_{K}^{0})\circ d_{K}^{n-1}\circ \dots \circ d_{K}^{2} \circ d_{K}^{1}=\omega_{K^{1}}$.
		
		$\vdots$
		
		\item [(n)]$T(k^{n-1}) \circ T(d_{K}^{n-2})\circ T(d_{K}^{n-3})\circ \dots \circ T(d_{K}^{0}) \circ d_{K}^{n-1}
		=T(d_{X}^{n-2})\circ T(k^{n-2})\circ T(d_{K}^{n-3})\circ \dots \circ T(d_{K}^{0}) \circ d_{K}^{n-1}
		=\dots
		=\omega_{X^{n-1}}\circ k^{n-1}
		=T(k^{n-1})\circ \omega_{K^{n-1}}$.
		
		Since $T(k^{n-1})$ is injective, we have $T(d_{K}^{n-2})\circ T(d_{K}^{n-3})\circ \dots \circ T(d_{K}^{0}) \circ d_{K}^{n-1}=\omega_{K^{n-1}}$.
	\end{itemize}
	
	Claim (b): The morphism $k=(k^{0},k^{1},\dots,k^{n-1}):K \rightarrow X$ is a required kenel of $\varphi$.
	
	As showen in the diagram. Let $M=(M^{j},d_{M}^{j})$ be an $n$-fold $(\mathcal{A},T)$-factorization of $\omega$, $g=(g^{0},g^{1},\dots,g^{n-1}):M \rightarrow X$ a morphism satisfying $\varphi\circ g=0$ in ${\rm Fact}_{n}(\mathcal{A},T,\omega)$. Then we have $\varphi^{j}\circ g^{j}=0$ for $j=0,1,\dots,n-1$. Therefore there exists a unique morphism $h^{j}:M^{j} \rightarrow K^{j}$ such that $k^{j}\circ h^{j}=g^{j}$ for $j=0,1,\dots,n-1$. Consequently, there exists a unique morphism $h:M \rightarrow K$ such that $k\circ h=g$. In deed, $h:M \rightarrow K$ is a morphism between $n$-fold $(\mathcal{A},T)$-factorizations of $\omega$. For any $0 \leq
	j \leq n-2$, $k^{j+1}\circ d_{K}^{j} \circ h^{j}=d_{X}^{j} \circ k^{j}\circ h^{j}=d_{X}^{j} \circ g^{j}=g^{j+1} \circ d_{M}^{j}=k^{j+1}\circ h^{j+1} \circ d_{M}^{j}$. Since $k^{j+1}$ is injective, $d_{K}^{j} \circ h^{j}=h^{j+1} \circ d_{M}^{j}$. Additionally, for $j=n-1$, $T(k^{0})\circ d_{K}^{n-1}\circ h^{n-1}=d_{X}^{n-1}\circ k^{n-1}\circ h^{n-1}=d_{X}^{n-1}\circ g^{n-1}=T(g^{0}) \circ d_{M}^{n-1}=T(k^{0})\circ T(h^{0}) \circ d_{M}^{n-1}$. Since $T(k^{0})$ is injective, $d_{K}^{n-1}\circ h^{n-1}=T(h^{0}) \circ d_{M}^{n-1}$.
 $$\scriptsize \xymatrix{M^{0}
		\ar[rr]^{d_{M}^{0}}
		\ar@/_/[ddr]_>(.6){g^{0}}
		\ar@{.>}[dr]^{h^{0}} & &M^{1}
		\ar[rr]^{d_{M}^{1}}
		\ar@/_/[ddr]_>(.6){g^{1}}
		\ar@{.>}[dr]^{h^{1}}& &M^{2}\ar[r]\ar@/_/[ddr]_>(.6){g^{2}}
		\ar@{.>}[dr]^{h^{2}}&\cdots\ar[r]&M^{n-1}\ar[rr]^{d_{M}^{n-1}}\ar@/_/[ddr]_>(.6){g^{n-1}}
		\ar@{.>}[dr]^{h^{n-1}}& & T(M^{0}) \ar@/_/[ddr]_>(.6){T(g^{0})}
		\ar@{.>}[dr]^{T(h^{0})} \\
		& K^{0} \ar[rr]^{d_{K}^{0}} \ar[d]^{k^{0}}&
		& K^{1} \ar[rr]^{d_{K}^{1}} \ar[d]^{k^{1}}& &K^{2}\ar[r]^{d_{K}^{2}}\ar[d]^{k^{2}}
		&\cdots\ar[r]&K^{n-1}\ar[d]^{k^{n-1}}\ar[rr]^{d_{K}^{n-1}}& &T(K^{0}) \ar[d]^{T(k^{0})} \\
		& X^{0} \ar[rr]^{d_{X}^{0}}\ar[d]^{\varphi^{0}}&
		&  X^{1} \ar[rr]^{d_{X}^{1}}\ar[d]^{\varphi^{1}}& &X^{2}\ar[r]^{d_{X}^{2}}\ar[d]^{ \varphi^{2}}
		&\cdots\ar[r]&X^{n-1}\ar[d]^{\varphi^{n-1}}\ar[rr]^{d_{X}^{n-1}}& &T(X^{0})\ar[d]^{T(\varphi^{0})}\\
		&Y^{0} \ar[rr]^{d_{Y}^{0}}& &Y^{1} \ar[rr]^{d_{Y}^{1}} & &Y^{2}\ar[r]^{d_{Y}^{2}}
		&\cdots\ar[r]&Y^{n-1}\ar[rr]^{d_{Y}^{n-1}}& &T(Y^{0}).   }$$
\end{proof}

Assume that $(\mathcal{A},\mathcal{E})$ is an exact category. Give a class $\tilde{\mathcal{E}}$ of composable morphism pairs $(l,p)$ in ${\rm Fact}_{n}(\mathcal{A},T,\omega)$ such that $(l^{0},p^{0}),(l^{1},p^{1}),\dots,(l^{n-1},p^{n-1})$ belong to $\mathcal{E}$. By Lemma \ref{lem1} and its dual, any element in $\tilde{\mathcal{E}}$ is a kernel-cokernel pair in ${\rm Fact}_{n}(\mathcal{A},T,\omega)$. Furthermore, If $\mathcal{A}$ is abelian, so is ${\rm Fact}_{n}(\mathcal{A},T,\omega)$.

\begin{lemma} \label{lem2}
 $({\rm Fact}_{n}(\mathcal{A},T,\omega),\tilde{\mathcal{E}})$ is 	an exact category.
\end{lemma}

\begin{proof}
	
	[E0$^{'}$] Obviously, $1_{(0,0)}:(0,0)\rightarrow (0,0)$ is a deflation.
	
	[E1] Suppose $p:Y\rightarrow Z,\ q:Z\rightarrow W$ are deflations in $\text{Fact}_{n}(\mathcal{A},T,\omega)$. Then for $0 \leq j \leq n-1$, we have $p^{j}:Y^{j}\rightarrow Z^{j},\ q^{j}:Z^{j}\rightarrow W^{j}$ are deflations in $\mathcal{A}$. So for any $j$, $q^{j}p^{j}$ is a deflation in $\mathcal{A}$. Furthermore, $qp$ is a deflation in $\text{Fact}_{n}(\mathcal{A},T,\omega)$.
	
	[E2] For a deflation $p:Y\rightarrow Z$ and a morphism $f: Z_{1} \rightarrow Z$ in $\text{Fact}_{n}(\mathcal{A},T,\omega)$, we can take the pullback of each component. As is shown in the following diagram.
$$\footnotesize \xymatrix @R=14pt@C=14pt{ Y_{1}^{0}\ar@{.>}[rr]^{d_{Y_{1}}^{0}}\ar@{.>}[dd]_(.3){f_{1}^{0}}\ar@{.>}[dr]^{p_{1}^{0}} & &Y_{1}^{1}\ar@{.>}[rr]^{d_{Y_{1}}^{1}}\ar@{.>}'[d]_(.6){f_{1}^{1}}[dd]\ar@{.>}[dr]^{p_{1}^{1}} & &Y_{1}^{2}\ar@{.>}[rr]\ar@{.>}'[d]_(.6){f_{1}^{2}}[dd]\ar@{.>}[dr]^{p_{1}^{2}} & &\cdots \ar@{.>}[rr] & &Y_{1}^{n-1}\ar@{.>}[rr]^{d_{Y_{1}}^{n-1}}\ar@{.>}'[d]_(.6){f_{1}^{n-1}}[dd]\ar@{.>}[dr]^{p_{1}^{n-1}}& &T(Y_{1}^{0})\ar@{.>}'[d]_(.6){T(f_{1}^{0})}[dd]\ar@{.>}[dr]^{T(p_{1}^{0})}\\
			 &Z_{1}^{0}\ar[rr]\ar[dd]^(.3){f^{0}} & &Z_{1}^{1}\ar[rr]\ar[dd]^(.3){f^{1}} & &Z_{1}^{2}\ar[rr]\ar[dd]^(.3){f^{2}} & &\cdots \ar[rr]& &Z_{1}^{n-1}\ar[rr]\ar[dd]^(.3){f^{n-1}} & &T(Z_{1}^{0})\ar[dd]^(.3){T(f^{0})}\\
			Y^{0}\ar[rr]\ar[dr]^{p^{0}} & &Y^{1}\ar[rr]\ar[dr]^{p^{1}} & &Y^{2}\ar[rr]\ar[dr]^{p^{2}} & &\cdots \ar[rr]& &Y^{n-1}\ar[rr]\ar[dr]^{p^{n-1}}& &T(Y^{0})\ar[dr]^{T(p^{0})}\\
			&Z^{0}\ar[rr]& &Z^{1}\ar[rr]& &Z^{2}\ar[rr] & & \cdots \ar[rr]& &Z^{n-1}\ar[rr] & &T(Z^{0})  .}$$
		
	For each pullback, we can obtain an exact sequence $0 \longrightarrow Y_{1}^{j} \stackrel{{\tiny \begin{bmatrix}
				p_{1}^{j} \\ -f_{1}^{j}
	\end{bmatrix}}}{\longrightarrow} Z_{1}^{j}\oplus Y^{j} \stackrel{\tiny\setlength{\arraycolsep}{1pt} \begin{bmatrix}
	f^{j} & p^{j}
\end{bmatrix}}{\longrightarrow} Z^{j}$. Meanwhile, for any $0 \leq j \leq n-2$, there exists a unique morphism $d_{Y_{1}}^{j}:Y_{1}^{j}\rightarrow Y_{1}^{j+1}$ satisfying $p_{1}^{j+1}\circ d_{Y_{1}}^{j}=d_{Z_{1}}^{j}\circ p_{1}^{j}$ and $f_{1}^{j+1}\circ d_{Y_{1}}^{j}=d_{Y}^{j}\circ f_{1}^{j}$; for $j=n-1$, there exists a unique morphism $d_{Y_{1}}^{n-1}:Y_{1}^{n-1}\rightarrow T(Y_{1}^{0})$ satisfying $T(p_{1}^{0})\circ d_{Y_{1}}^{n-1}=d_{Z_{1}}^{n-1}\circ p_{1}^{n-1}$ and $T(f_{1}^{0})\circ d_{Y_{1}}^{n-1}=d_{Y}^{n-1}\circ f_{1}^{n-1}$. Then we can obtain an entire commutative diagram which each row is exact.
$$\xymatrix{0 \ar[r] &Y_{1}^{0} \ar[r] \ar[d]^{d_{Y_{1}}^{0}} &Z_{1}^{0}\oplus Y^{0} \ar[r] \ar[d]^{d_{Z_{1}}^{0}\oplus d_{Y}^{0}} &Z^{0} \ar[d]^{d_{Z}^{0}}\\
		0 \ar[r] &Y_{1}^{1} \ar[r] \ar[d]^{d_{Y_{1}}^{1}} &Z_{1}^{1}\oplus Y^{1} \ar[r] \ar[d]^{d_{Z_{1}}^{1}\oplus d_{Y}^{1}} &Z^{1} \ar[d]^{d_{Z}^{1}}\\
		0 \ar[r] &Y_{1}^{2} \ar[r] \ar@{.>}[d] &Z_{1}^{2}\oplus Y^{2} \ar[r] \ar@{.>}[d] &Z^{2} \ar@{.>}[d]\\
		0 \ar[r] &Y_{1}^{n-1} \ar[r] \ar[d]^{d_{Y_{1}}^{n-1}} &Z_{1}^{n-1}\oplus Y^{n-1} \ar[r] \ar[d]^{d_{Z_{1}}^{n-1}\oplus d_{Y}^{n-1}} &Z^{n-1} \ar[d]^{d_{Z}^{n-1}}\\
		0 \ar[r] &T(Y_{1}^{0}) \ar[r] &T(Z_{1}^{0})\oplus T(Y^{0}) \ar[r]   &T(Z^{0})  .}$$

Since $\begin{bmatrix}
	p_{1}^{j}\\ -f_{1}^{j}
\end{bmatrix}$ is injective for any $0 \leq j \leq n-1$, we can verify that $Y_{1}$ is an $n$-fold $(\mathcal{A},T)$-factorization of $\omega$.
Additionally, the morphisms $f_{1}=(f_{1}^{0},f_{1}^{1},\dots,f_{1}^{n-1}),\ p_{1}=(p_{1}^{0},p_{1}^{1},\dots,p_{1}^{n-1})$ are the morphisms between $n$-fold $(\mathcal{A},T)$-factorizations of $\omega$. Thus, the pullback of $p$ and $f$ exists and is given componentwise.

	[E1$^{\text{op}}$] and [E2$^{\text{op}}$] are dual.
\end{proof}

\begin{lemma} \label{lem3}
	We have adjoint pairs $(\theta^{0},{\rm pr}^{0}),(\theta^{1},{\rm pr}^{1}),({\rm pr}^{n-1},\theta^{0}),({\rm pr}^{n-1}S,\theta^{1})$.
\end{lemma}

\begin{proof}
	For any object $C \in \mathcal{A}, X \in {\rm Fact}_{n}(\mathcal{A},T,\omega)$.
	\begin{itemize}
		\item [(1)]
		Define a morphism
		\begin{center}
		$\Phi_{C,X}:{\rm Hom}_{\mathcal{A}}(C,X^{0}) \rightarrow {\rm Hom}_{{\rm Fact}_{n}(\mathcal{A},T,\omega)}(\theta^{0}(C),X),\ \ \ \ \ g \mapsto (g^{0},g^{1},\dots,g^{n-1})$,\\
		\end{center}
		where $g^{0}=g$ and  $g^{j}=d_{X}^{j-1}\circ \dots \circ d_{X}^{0}\circ g$ for $j=1,2,\dots n-1$.
$$\xymatrix{C \ar[r]^{\rm Id_{C}}\ar[d]^{g^{0}=g}&C \ar[r]^{\rm Id_{C}}\ar[d]^{g^{1}} &C
				\ar[r]\ar[d]^{g^{2}} &\cdots
				\ar[r]^{\rm Id_{C}} &C \ar[r]^{\omega_{C}}\ar[d]^{g{n-1}} &T(C) \ar[d]^{T(g^{0})}\\
				X^{0} \ar[r]^{d_{X}^{0}}&X^{1} \ar[r]^{d_{X}^{1}} &X^{2} \ar[r] &\cdots
				\ar[r]^{d_{X}^{n-2}} &X^{n-1} \ar[r]^{d_{X}^{n-1}} &T(X^{0}) .}$$
		
		The definition is reasonable. We just need to verify that the rightmost square is commutative.
		\begin{center}
			$T(g^{0}) \circ \omega_{C}=\omega_{X^{0}}\circ g^{0}=(d_{X}^{n-1}\circ \dots \circ d_{X}^{0})\circ g^{0}=d_{X}^{n-1}\circ g^{n-1}$.
		\end{center}
		
		We can verify $\Phi$ is a natural isomorphism. Firstly, define a morphism
		\begin{center}
		$\Psi:{\rm Hom}_{\rm {Fact}_{n}(\mathcal{A},T,\omega)}(\theta^{0}(C),X) \rightarrow {\rm Hom}_{\mathcal{A}}(C,X^{0}),\ \ \ \ \ (g^{0},g^{1},\dots,g^{n-1}) \mapsto g^{0}$.\\
		\end{center}
		 By the commutativity of the morphisms between $n$-fold $(\mathcal{A},T)$-factorizations, once $g^{0}$ is determined, the other morphisms will also be correspondingly uniquely determined. Thus, we can verify that $\Phi$ and $\Psi$ are inverses of each other. Additionally, for each object $C_{1} \in \mathcal{A}, X_{1} \in \text{Fact}_{n}(\mathcal{A},T,\omega), f:C_{1} \rightarrow C, h=(h^{0},h^{1},\dots,h^{n-1}):X\rightarrow X_{1}$, the following diagram is commutative. Let $g$ be a morphism in ${\rm Hom}_{\mathcal{A}}(C,X^{0})$.
$$\xymatrix{{\rm Hom}_{\mathcal{A}}(C,X^{0}) \ar[r]^(.4){\Phi_{C,X}}\ar[d]^{f^{*}h^{0}_{*}}& {\rm Hom}_{{\rm Fact}_{n}(\mathcal{A},T,\omega)}(\theta^{0}(C),X) \ar[d]^{(\theta^{0}(f))^{*}h_{*}}\\
				{\rm Hom}_{\mathcal{A}}(C_{1},X_{1}^{0}) \ar[r]^(.4){\Phi_{C_{1},X_{1}}}&{\rm Hom}_{{\rm Fact}_{n}(\mathcal{A},T,\omega)}(\theta^{0}(C_{1}),X_{1})  .}$$
		
		Anticlockwise:
		$\Phi_{C_{1},X_{1}}\circ (f^{*}h^{0}_{*})(g)=\Phi_{C_{1},X_{1}}(h^{0}gf)\\
		=(h^{0}gf, d_{X_{1}}^{0} h^{0}gf, \dots, d_{X_{1}}^{n-2} \dots  d_{X_{1}}^{0} h^{0}gf)$.
		
		Clockwise:
		$((\theta^{0}(f))^{*}h_{*})\circ \Phi_{C,X}(g)
		=((\theta^{0}(f))^{*}h_{*})(g,d_{X}^{0} g, \dots, d_{X}^{n-2} \dots d_{X}^{0}g)\\
		=(h^{0}gf, h^{1} d_{X}^{0} gf, \dots, h^{n-1} d_{X}^{n-2} \dots  d_{X}^{0} gf)\\
		=(h^{0}gf, d_{X_{1}}^{0}h^{0}gf, \dots, d_{X_{1}}^{n-2} \dots  d_{X_{1}}^{0}h^{0}gf)$.
		
		Therefore $(\theta^{0},{\rm pr}^{0})$ is an adjoint pair.
	
	\item [(2)]
		Define a morphism
		\begin{center}
			$\Phi_{C,X}:\text{Hom}_{\mathcal{A}}(C,X^{1}) \rightarrow \text{Hom}_{\text{Fact}_{n}(\mathcal{A},T,\omega)}(\theta^{1}(C),X),\ \ \ \ \ g \mapsto (g^{0},g^{1},\dots,g^{n-1})$,\\
		\end{center}
		where $g^{0}=\epsilon_{X^{0}}\circ T^{-1}(d_{X}^{n-1}\circ \dots \circ d_{X}^{1}\circ g)$,
			$g^{1}=g$ and
			$g^{j}=d_{X}^{j-1}\circ d_{X}^{j-2}\circ \dots \circ d_{X}^{1}\circ g$ for $j=2,\dots, n-1$.
$$\xymatrix{T^{-1}C \ar[r]^(.6){\omega^{(-1)}_{C}}\ar[d]^{g^{0}}&C \ar[r]^{\rm Id_{C}}\ar[d]^{g^{1}=g} &C
				\ar[r]\ar[d]^{g^{2}} &\cdots
				\ar[r]^{\rm Id_{C}} &C \ar[r]^{\eta_{C}}\ar[d]^{g{n-1}} &TT^{-1}C \ar[d]^{T(g^{0})}\\
				X^{0} \ar[r]^{d_{X}^{0}}&X^{1} \ar[r]^{d_{X}^{1}} &X^{2} \ar[r] &\cdots
				\ar[r]^{d_{X}^{n-2}} &X^{n-1} \ar[r]^{d_{X}^{n-1}} &T(X^{0})   .}$$
		
		The definition is reasonable. We just need to verify that the leftmost and rightmost squares are commutative.
		
		$d_{X}^{0}\circ g^{0}=d_{X}^{0}\circ \epsilon_{X^{0}}\circ T^{-1}(d_{X}^{n-1}\circ \dots \circ d_{X}^{1}\circ g)
		=\epsilon_{X^{1}}\circ T^{-1}T(d_{X}^{0})\circ T^{-1}(d_{X}^{n-1}\circ \dots \circ d_{X}^{1}\circ g)
		=\epsilon_{X^{1}}\circ T^{-1}(\omega_{X^{1}}\circ g)
		=\epsilon_{X^{1}}\circ T^{-1}(T(g)\circ \omega_{C})
		=\epsilon_{X^{1}}\circ T^{-1}T(g)\circ T^{-1}(\omega_{C})
		=g\circ \epsilon_{C} \circ T^{-1}\omega_{C}
		=g\circ \omega_{C}^{(-1)}$.

		$T(g^{0})\circ \eta_{C}=T(\epsilon_{X^{0}}\circ T^{-1}(d_{X}^{n-1}\circ \dots \circ d_{X}^{1}\circ g))\circ \eta_{C}
		=T(\epsilon_{X^{0}})\circ TT^{-1}(d_{X}^{n-1}\circ \dots \circ d_{X}^{1}\circ g)\circ \eta_{C}
		=T(\epsilon_{X^{0}})\circ \eta_{T(X^{0})} \circ (d_{X}^{n-1}\circ \dots \circ d_{X}^{1}\circ g)
		=d_{X}^{n-1}\circ \dots \circ d_{X}^{1}\circ g$.
		
		We can verify $\Phi$ is a natural isomorphism. Define a morphism
		\begin{center}
			$\Psi_{C,X}:{\rm Hom}_{{\rm Fact}_{n}(\mathcal{A},T,\omega)}(\theta^{1}(C),X)\rightarrow \text{Hom}_{\mathcal{A}}(C,X^{1}) ,\ \ \ \ \ (g^{0},g^{1},\dots,g^{n-1}) \mapsto g^{1}$.\\
		\end{center}
		In fact, when $g^{1}:C \rightarrow X^{1}$ is fixed, the other morphisms are correspondingly uniquely determined by the commutativity of the morphisms between $n$-fold $(\mathcal{A},T)$-factorizations. When $2 \leq j \leq n-1$, this is obvious. We mainly prove that $g^{0}$ is uniquely determined by $g^{1}$.
		\begin{center}
			 $T(g^{0}) \circ \eta_{C}=d_{X}^{n-1}\circ g^{n-1}=d_{X}^{n-1}d_{X}^{n-2}\circ \dots \circ d_{X}^{1}g^{1}$.
		\end{center}
		\begin{center}
			$T(g^{0})=(d_{X}^{n-1}d_{X}^{n-2}\circ \dots \circ d_{X}^{1}g^{1}) \eta_{C}^{-1}$.
		\end{center} By the naturality of $\epsilon$, $g^{0}\circ \epsilon_{T^{-1}C}=\epsilon_{X^{0}}\circ T^{-1}T(g^{0})=\epsilon_{X^{0}}\circ T^{-1}(d_{X}^{n-1}d_{X}^{n-2}\circ \dots \circ d_{X}^{1}g^{1}\eta_{C}^{-1})$.
		
		$g^{0}=\epsilon_{X^{0}}\circ T^{-1}(d_{X}^{n-1}d_{X}^{n-2}\circ \dots \circ d_{X}^{1}g^{1}\eta_{C}^{-1})\circ \epsilon_{T^{-1}C}^{-1}\\
		=\epsilon_{X^{0}}\circ T^{-1}(d_{X}^{n-1}d_{X}^{n-2}\circ \dots \circ d_{X}^{1}g^{1})\circ T^{-1}(\eta_{C}^{-1})\circ \epsilon_{T^{-1}C}^{-1}\\
		=\epsilon_{X^{0}}\circ T^{-1}(d_{X}^{n-1}d_{X}^{n-2}\circ \dots \circ d_{X}^{1}g^{1}).$
		
		Therefore, $\Phi$ and $\Psi$ are inverses of each other.
		Additionally, for each object $C_{1} \in \mathcal{A}, X_{1} \in \text{Fact}_{n}(\mathcal{A},T,\omega), f:C_{1} \rightarrow C, h=(h^{0},h^{1},\dots,h^{n-1}):X\rightarrow X_{1}$. The following diagram is commutative. Let $g$ be a morphism in ${\rm Hom}_{\mathcal{A}}(C,X^{1})$.
$$\xymatrix{{\rm Hom}_{\mathcal{A}}(C,X^{1}) \ar[r]^(.4){\Phi_{C,X}}\ar[d]^{f^{*}h^{1}_{*}}& {\rm Hom}_{{\rm Fact}_{n}(\mathcal{A},T,\omega)}(\theta^{1}(C),X) \ar[d]^{(\theta^{1}(f))^{*}h_{*}}\\
				{\rm Hom}_{\mathcal{A}}(C_{1},X_{1}^{1}) \ar[r]^(.4){\Phi_{C_{1},X_{1}}}&{\rm Hom}_{{\rm Fact}_{n}(\mathcal{A},T,\omega)}(\theta^{1}(C_{1}),X_{1})   .}$$
		
		Anticlockwise: $\Phi_{C_{1},X_{1}}\circ (f^{*}h^{1}_{*})(g) =\Phi_{C_{1},X_{1}}(h^{1}gf)\\
		=(\epsilon_{X_{1}^{0}}\circ T^{-1}(d_{X_{1}}^{n-1} \dots d_{X_{1}}^{1}  h^{1}gf), h^{1}gf, d_{X_{1}}^{1}  h^{1}gf, \dots, d_{X_{1}}^{n-2} \dots d_{X_{1}}^{1} h^{1}gf)$.
		
		Clockwise: $((\theta^{1}(f))^{*}h_{*})\circ \Phi_{C,X}(g)\\
		=(\theta^{1}(f))^{*}h_{*})(\epsilon_{X^{0}}\circ T^{-1}(d_{X}^{n-1} \dots  d_{X}^{1}  g),g,d_{X}^{1} g,\dots,d_{X}^{n-2} \dots  d_{X}^{1}  g )\\
		=(h^{0}(\epsilon_{X^{0}}\circ T^{-1}(d_{X}^{n-1} \dots  d_{X}^{1}  g)) T^{-1}(f),h^{1}gf,h^{2}d_{X}^{1} g f,\dots,h^{n-1}d_{X}^{n-2}  \dots d_{X}^{1}  gf )\\
		=(\epsilon_{X_{1}^{0}}\circ T^{-1}(d_{X_{1}}^{n-1}  \dots  d_{X_{1}}^{1}  h^{1}gf), h^{1}gf, d_{X_{1}}^{1}  h^{1}gf, \dots, d_{X_{1}}^{n-2}  \dots   d_{X_{1}}^{1}  h^{1}gf)$.
		
		Therefore $(\theta^{1},{\rm pr}^{1})$ is an adjoint pair.
		
	\item [(3)]
		Define a morphism
		\begin{center}
			$\Phi_{X,C}:{\rm Hom}_{\mathcal{A}}(X^{n-1},C) \rightarrow {\rm Hom}_{\text{Fact}_{n}(\mathcal{A},T,\omega)}(X,\theta^{0}(C)),\ \ \ \ \ g \mapsto (g^{0},g^{1},\dots,g^{n-1})$,\\
		\end{center}
		where $g^{j}=g\circ d_{X}^{n-2}\circ \dots \circ d_{X}^{j}$ for $0 \leq j \leq n-2$ and
			$g^{n-1}=g$.
$$\xymatrix{X^{0} \ar[r]^{d_{X}^{0}}\ar[d]^{g^{0}}&X^{1} \ar[r]^{d_{X}^{1}}\ar[d]^{g^{1}} &X^{2}
				\ar[r]\ar[d]^{g^{2}} &\cdots
				\ar[r]^{d_{X}^{n-2}} &X^{n-1} \ar[r]^{d_{X}^{n-1}}\ar[d]^{g{n-1}=g} &T(X^{0}) \ar[d]^{T(g^{0})}\\
				C \ar[r]^{\rm Id_{C}}&C \ar[r]^{\rm Id_{C}} &C
				\ar[r]&\cdots
				\ar[r]^{\rm Id_{C}} &C \ar[r]^{\omega_{C}}&T(C)   .}$$
		The definition is reasonable. We just need to verify that the rightmost square is commutative.
		
		$T(g^{0})\circ d_{X}^{n-1}
		=T(g\circ d_{X}^{n-2}\circ \dots \circ d_{X}^{0})\circ d_{X}^{n-1}
		=T(g)\circ T(d_{X}^{n-2})\circ T(d_{X}^{n-1})\circ \dots \circ T(d_{X}^{0})\circ d_{X}^{n-1}
		=T(g)\circ \omega_{X^{n-1}}
		=\omega_{C} \circ g$.
		
		We can verify $\Phi$ is a natural isomorphism. Define a morphism
		\begin{center}
			$\Psi_{X,C}:{\rm Hom}_{{\rm Fact}_{n}(\mathcal{A},T,\omega)}(X,\theta^{0}(C)) \rightarrow {\rm Hom}_{\mathcal{A}}(X^{n-1},C),\ \ \ \ \ (g^{0},g^{1},\dots,g^{n-1}) \mapsto g^{n-1}$.
		\end{center}
		In fact, for any morphism $(g^{0},g^{1},\dots,g^{n-1})$ in ${\rm Hom}_{{\rm Fact}_{n}(\mathcal{A},T,\omega)}(X,\theta^{0}(C))$, the following equality is always maintained.
		\begin{center}
			$g^{j}=g^{n-1}d_{X}^{n-2}\circ \dots \circ d_{X}^{j}$, for any $0 \leq j \leq n-2$.
		\end{center}
		This indicates that when the morphism $g^{n-1}:X^{n-1}\rightarrow C$ is fixed, the other morphisms are correspondingly uniquely determined by the commutativity of the morphisms between $n$-fold $(\mathcal{A},T)$-factorizations. Thus, $\Phi$ and $\Psi$ are inverses of each other.
		
		For each object $C_{1} \in \mathcal{A}, X_{1} \in \text{Fact}_{n}(\mathcal{A},T,\omega), f:C \rightarrow C_{1}, h=(h^{0},h^{1},\dots,h^{n-1}):X_{1}\rightarrow X$. The following diagram is commutative. Let $g$ be a morphism in ${\rm Hom}_{\mathcal{A}}(X^{n-1},C)$.
$$\xymatrix{{\rm Hom}_{\mathcal{A}}(X^{n-1},C) \ar[r]^(.4){\Phi_{X,C}}\ar[d]^{(h^{n-1})^{*}f_{*}}& {\rm Hom}_{{\rm Fact}_{n}(\mathcal{A},T,\omega)}(X,\theta^{0}(C)) \ar[d]^{h^{*}(\theta^{0}(f))_{*}}\\
				{\rm Hom}_{\mathcal{A}}(X_{1}^{n-1},C_{1}) \ar[r]^(.4){\Phi_{X_{1},C_{1}}}&{\rm Hom}_{{\rm Fact}_{n}(\mathcal{A},T,\omega)}(X_{1},\theta^{0}(C_{1}))   .}$$
		
		Anticlockwise: $\Phi_{X_{1},C_{1}}\circ (h^{n-1})^{*}f_{*}(g) =\Phi_{X_{1},C_{1}}(fgh^{n-1}) \\
		=(fgh^{n-1} d_{X_{1}}^{n-2} \dots  d_{X_{1}}^{0},fgh^{n-1} d_{X_{1}}^{n-2} \dots d_{X_{1}}^{1}, \dots, fgh^{n-1})$.
		
		Clockwise: $(h^{*}(\theta^{0}(f))_{*}\circ \Phi_{X,C}(g)\\
		=(h^{*}(\theta^{0}(f))_{*}(g d_{X}^{n-2} \dots d_{X}^{0},g d_{X}^{n-2} \dots d_{X}^{1}, \dots, g)\\
		=(fg d_{X}^{n-2} \dots d_{X}^{0}h^{0},fg d_{X}^{n-2} \dots d_{X}^{1}h^{1}, \dots, fgh^{n-1})\\
		=(fgh^{n-1} d_{X_{1}}^{n-2} \dots d_{X_{1}}^{0},fgh^{n-1} d_{X_{1}}^{n-2} \dots d_{X_{1}}^{1}, \dots, fgh^{n-1})$.
		
		Therefore $({\rm pr}^{n-1},\theta^{0})$ is an adjoint pair.

	\item [(4)]
		Define a morphism
		\begin{center}
			$\Phi_{X,C}:{\rm Hom}_{\mathcal{A}}(T(X^{0}),C) \rightarrow {\rm Hom}_{{\rm Fact}_{n}(\mathcal{A},T,\omega)}(X,\theta^{1}(C)),\ \ \ \ \ g \mapsto (g^{0},g^{1},\dots,g^{n-1})$,\\
		\end{center}
		where $g^{0}=T^{-1}(g)\circ \epsilon_{X^{0}}^{-1}$ and
			$g^{j}=g\circ d_{X}^{n-1}\circ \dots \circ d_{X}^{j}$ for $1 \leq j \leq n-1$.
$$\xymatrix{X^{0} \ar[r]^{d_{X}^{0}}\ar[d]^{g^{0}}&X^{1} \ar[r]^{d_{X}^{1}}\ar[d]^{g^{1}} &X^{2}
				\ar[r]\ar[d]^{g^{2}} &\cdots
				\ar[r]^{d_{X}^{n-2}} &X^{n-1} \ar[r]^{d_{X}^{n-1}}\ar[d]^{g{n-1}} &T(X^{0}) \ar[d]^{T(g^{0})}\\
				T^{-1}C \ar[r]^(.6){\omega_{C}^{(-1)}}&C \ar[r]^{\rm Id_{C}} &C
				\ar[r]&\cdots
				\ar[r]^{\rm Id_{C}} &C \ar[r]^{\eta_{C}}&TT^{-1}C    .}$$
			
		The definition is reasonable. We just need to verify that the leftmost and the rightmost squares are commutative.
		
		$\omega_{C}^{(-1)} \circ g^{0}
		=\omega_{C}^{(-1)} \circ T^{-1}(g)\circ \epsilon_{X^{0}}^{-1}
		=g\circ \omega_{T(X^{0})}^{(-1)}\circ \epsilon_{X^{0}}^{-1}
		=g\circ \omega_{X^{0}}
		=g\circ d_{X}^{n-1}\circ \dots \circ d_{X}^{0}
		=g^{1}\circ d_{X}^{0}$.
		
		$T(g^{0})\circ d_{X}^{n-1}
		=T(T^{-1}(g)\circ \epsilon_{X^{0}}^{-1})\circ d_{X}^{n-1}
		=TT^{-1}(g)\circ T(\epsilon_{X^{0}}^{-1}) \circ d_{X}^{n-1}
		=TT^{-1}(g)\circ \eta_{T(X^{0})} \circ d_{X}^{n-1}
		=\eta_{C} \circ g\circ d_{X}^{n-1}
		=\eta_{C} \circ g^{n-1}$.
		
		We can verify $\Phi$ is a natural isomorphism.
		Define a morphism
		\begin{center}
			$\Psi_{X,C}:{\rm Hom}_{{\rm Fact}_{n}(\mathcal{A},T,\omega)}(X,\theta^{1}(C)) \rightarrow {\rm Hom}_{\mathcal{A}}(T(X^{0}),C),\ (g^{0},g^{1},\dots,g^{n-1}) \mapsto \eta_{C}^{-1}\circ T(g^{0})$,\\
		\end{center}
		with the equalities $\eta_{C}\circ g^{n-1}=T(g^{0})\circ d_{X}^{n-1}, g^{j}=g^{n-1}d_{X}^{n-2}\dots d_{X}^{j}$ for $1 \leq j \leq n-2$ and $\omega_{C}^{(-1)} \circ g^{0}=g^{1} d_{X}^{0}$.
		
		For any morphism $g \in {\rm Hom}_{\mathcal{A}}(T(X^{0}),C)$,
		
		$\Psi \Phi(g)=\Psi(T^{-1}(g)\circ \epsilon_{X^{0}}^{-1},g\circ d_{X}^{n-1}\circ \dots \circ d_{X}^{1},\dots,g\circ d_{X}^{n-1})\\
		=\eta_{C}^{-1}\circ T(T^{-1}(g)\circ \epsilon_{X^{0}}^{-1})\\
		=g \circ \eta_{T(X^{0})}^{-1} \circ T(\epsilon_{X^{0}}^{-1})=g.$
		
		For any morphism $(g^{0},g^{1},\dots, g^{n-1}) \in {\rm Hom}_{{\rm Fact}_{n}(\mathcal{A},T,\omega)}(X,\theta^{1}(C))$,
		
		$\Phi\Psi(g^{0},g^{1},\dots, g^{n-1})=\Phi(\eta_{C}^{-1}\circ T(g^{0}))\\
		=(T^{-1}(\eta_{C}^{-1}\circ T(g^{0}))\circ \epsilon_{X^{0}}^{-1},\eta_{C}^{-1}\circ T(g^{0})d_{X}^{n-1}\dots d_{X}^{1},\dots ,\eta_{C}^{-1}\circ T(g^{0})d_{X}^{n-1})\\
		=(\epsilon_{T^{-1}C}\circ T^{-1}T(g^{0})\circ \epsilon_{X^{0}}^{-1},\eta_{C}^{-1}\circ T(g^{0})d_{X}^{n-1}\dots d_{X}^{1},\dots ,\eta_{C}^{-1}\circ T(g^{0})d_{X}^{n-1})\\
		=(g^{0}\circ \epsilon_{X^{0}}\epsilon_{X^{0}}^{-1},g^{n-1}d_{X}^{n-2}\dots d_{X}^{1},\dots,g^{n-1})\\
		=(g^{0},g^{1},\dots, g^{n-1})$.
		
		For each object $C_{1} \in \mathcal{A}, X_{1} \in \text{Fact}_{n}(\mathcal{A},T,\omega), f:C \rightarrow C_{1}, h=(h^{0},h^{1},\dots,h^{n-1}):X_{1}\rightarrow X$. The following diagram is commutative. Let $g$ be a morphism in ${\rm Hom}_{\mathcal{A}}(T(X^{0}),C)$.
$$\xymatrix{{\rm Hom}_{\mathcal{A}}(T(X^{0}),C) \ar[r]^(.4){\Phi_{X,C}}\ar[d]^{T(h^{0})^{*}f_{*}}& {\rm Hom}_{{\rm Fact}_{n}(\mathcal{A},T,\omega)}(X,\theta^{1}(C)) \ar[d]^{h^{*}(\theta^{1}(f))_{*}}\\
				{\rm Hom}_{\mathcal{A}}(T(X_{1}^{0}),C_{1}) \ar[r]^(.4){\Phi_{X_{1},C_{1}}}&{\rm Hom}_{{\rm Fact}_{n}(\mathcal{A},T,\omega)}(X_{1},\theta^{1}(C_{1}))      .}$$
		
		Anticlockwise: $\Phi_{X_{1},C_{1}}\circ (T(h^{0})^{*}f_{*})(g) =\Phi_{X_{1},C_{1}}(fgT(h^{0})) \\
		=(T^{-1}(fgT(h^{0})) \epsilon_{X_{1}^{0}}^{-1},fgT(h^{0}) d_{X_{1}}^{n-1} \dots  d_{X_{1}}^{1}, \dots, fgT(h^{0}) d_{X_{1}}^{n-1})$.
		
		Clockwise: $(h^{*}(\theta^{1}(f))_{*})\circ \Phi_{X,C}(g)\\
		=(h^{*}(\theta^{1}(f))_{*})(T^{-1}(g) \epsilon_{X^{0}}^{-1},g d_{X}^{n-1} \dots  d_{X}^{1}, \dots, g d_{X}^{n-1})\\
		=(T^{-1}(f) T^{-1}(g)  \epsilon_{X^{0}}^{-1}  h^{0},f  g  d_{X}^{n-1}  \dots  d_{X}^{1}  h^{1}, \dots, f g  d_{X}^{n-1}  h^{n-1})\\
		=((T^{-1}(fgT(h^{0}))  \epsilon_{X_{1}^{0}}^{-1},fgT(h^{0}) d_{X_{1}}^{n-1} \dots   d_{X_{1}}^{1}, \dots, fgT(h^{0})  d_{X_{1}}^{n-1})$.
		
		Therefore $({\rm pr}^{n-1}S,\theta^{1})$ is an adjoint pair.
\end{itemize}
\end{proof}

In fact, regarding functors $\theta^{j}$, ${\rm pr}^{j}$ and adjoint pairs, there is a more detailed introduction in \cite{SZ}. What we have mentioned in this paper happens to be some specific cases among them.

\begin{lemma} \label{lem4}
	Let $T$ be an exact functor. For any projective objects $P$, $Q$ and injective objects $I$, $J$ in $\mathcal{A}$, the following statements hold.
	\begin{itemize}
		\item [\rm (1)]
		$\theta^{0}(P)$ and $\theta^{1}(Q)$ are projective.
		\item [\rm (2)]
		$\theta^{0}(I)$ and $\theta^{1}(J)$ are injective.
	\end{itemize}
\end{lemma}

\begin{proof}
	\begin{itemize}
		\item [(1)] Let $A\rightarrow B\rightarrow C$ be a conflation in ${\rm Fact}_{n}(\mathcal{A},T,\omega)$. In detail, for each $j$, the sequence $A^{j}\rightarrow B^{j}\rightarrow C^{j}$ is a conflation in $\mathcal{A}$. By the adjoint pairs $(\theta^{0},{\rm pr}^{0})$ and $(\theta^{1},{\rm pr}^{1})$ of Lemma \ref{lem3}, we have the following sequences
$$\small \xymatrix{{\rm Hom}_{{\rm Fact}_{n}(\mathcal{A},T,\omega)}(\theta^{0}(P),A) \ar[r]\ar[d]^{\cong}&{\rm Hom}_{{\rm Fact}_{n}(\mathcal{A},T,\omega)}(\theta^{0}(P),B) \ar[r]\ar[d]^{\cong} &{\rm Hom}_{{\rm Fact}_{n}(\mathcal{A},T,\omega)}(\theta^{0}(P),C)
				\ar[d]^{\cong}\\
				{\rm Hom}_{\mathcal{A}}(P,A^{0}) \ar[r]&{\rm Hom}_{\mathcal{A}}(P,B^{0}) \ar[r] &{\rm Hom}_{\mathcal{A}}(P,C^{0})   .}$$
		
$$\small \xymatrix{{\rm Hom}_{{\rm Fact}_{n}(\mathcal{A},T,\omega)}(\theta^{1}(Q),A) \ar[r]\ar[d]^{\cong}&{\rm Hom}_{{\rm Fact}_{n}(\mathcal{A},T,\omega)}(\theta^{1}(Q),B) \ar[r]\ar[d]^{\cong} &{\rm Hom}_{{\rm Fact}_{n}(\mathcal{A},T,\omega)}(\theta^{1}(Q),C)
				\ar[d]^{\cong}\\
				{\rm Hom}_{\mathcal{A}}(Q,A^{1}) \ar[r]&{\rm Hom}_{\mathcal{A}}(Q,B^{1}) \ar[r] &{\rm Hom}_{\mathcal{A}}(Q,C^{1})   .}$$

		Since $P$, $Q$ are projective, the second rows are conflations. Hence the first rows are also conflations. This yields that $\theta^{0}(P)$ and $\theta^{1}(Q)$ are projective.
		\item [(2)]
		Let $A\rightarrow B\rightarrow C$ be a conflation in ${\rm Fact}_{n}(\mathcal{A},T,\omega)$. In detail, for each $j$, the sequence $A^{j}\rightarrow B^{j}\rightarrow C^{j}$ is a conflation in $\mathcal{A}$. By the adjoint pairs $({\rm pr}^{n-1},\theta^{0})$ and $({\rm pr}^{n-1}S,\theta^{1})$ of Lemma \ref{lem3}, we have the following sequences
$$\small \xymatrix{{\rm Hom}_{\mathcal{A}}(C^{n-1},I) \ar[r]\ar[d]^{\cong}&{\rm Hom}_{\mathcal{A}}(B^{n-1},I)  \ar[r]\ar[d]^{\cong} &{\rm Hom}_{\mathcal{A}}(A^{n-1},I)
				\ar[d]^{\cong}\\
				{\rm Hom}_{{\rm Fact}_{n}(\mathcal{A},T,\omega)}(C,\theta^{0}(I)) \ar[r]&{\rm Hom}_{{\rm Fact}_{n}(\mathcal{A},T,\omega)}(B,\theta^{0}(I)) \ar[r] &{\rm Hom}_{{\rm Fact}_{n}(\mathcal{A},T,\omega)}(A,\theta^{0}(I))  .}$$
		
$$\small \xymatrix{{\rm Hom}_{\mathcal{A}}(T(C^{0}),J) \ar[r]\ar[d]^{\cong}&{\rm Hom}_{\mathcal{A}}(T(B^{0}),J)  \ar[r]\ar[d]^{\cong} &{\rm Hom}_{\mathcal{A}}(T(A^{0}),J)
				\ar[d]^{\cong}\\
				{\rm Hom}_{{\rm Fact}_{n}(\mathcal{A},T,\omega)}(C,\theta^{1}(J)) \ar[r]&{\rm Hom}_{{\rm Fact}_{n}(\mathcal{A},T,\omega)}(B,\theta^{1}(J)) \ar[r] &{\rm Hom}_{{\rm Fact}_{n}(\mathcal{A},T,\omega)}(A,\theta^{1}(J))   .}$$

		Since $I$, $J$ are injective and $T$ is an exact functor, the first rows are conflations. Hence the second rows are also conflations. This yields that $\theta^{0}(I)$ and $\theta^{1}(J)$ are injective.
	\end{itemize}
\end{proof}

\begin{lemma} \label{lem5}
	Assume that $\mathcal{A}$ has enough projective objects. Then so does ${\rm Fact}_{n}(\mathcal{A},T,\omega)$. Moreover, an $n$-fold $(\mathcal{A},T)$-factorization $X$ of $\omega$ is projective if and only if it is a direct summand of $\theta^{0}(P^{0}) \oplus \theta^{1}(P^{1})$ for some projective objects $P^{0}$ and $P^{1}$ in $\mathcal{A}$.
\end{lemma}
\begin{proof}
	Let $X$ be any $n$-fold $(\mathcal{A},T)$-factorization $X$ of $\omega$. Then we can take two deflations $a:P^{0} \rightarrow X^{0}$, $b:P^{1} \rightarrow X^{1}$ which $P^{0}$, $P^{1}$ are projective objects in $\mathcal{A}$. By adjunctions $(\theta^{0},{\rm pr}^{0})$ and $(\theta^{1},{\rm pr}^{1})$, we can obtain two morphisms $a^{'}:\theta^{0}(P^{0}) \rightarrow X$, $b^{'}:\theta^{1}(P^{1}) \rightarrow X$. Take the pullback of $a^{'}$ and $b^{'}$, we can obtain a deflation $\theta^{0}(P^{0})\oplus \theta^{1}(P^{1}) \rightarrow X$. By Lemma \ref{lem4} (1), we have $\theta^{0}(P^{0})$ and $\theta^{1}(P^{1})$ are projective. This proves that $\text{Fact}_{n}(\mathcal{A},T,\omega)$ has enough projective objects.
	
	Moreover, if $X$ is projective, the deflation splits. Therefore this yields the final conclusion.
\end{proof}	

\begin{lemma} \label{lem5 dual}
	Assume that $\mathcal{A}$ has enough injective objects. Then so does ${\rm Fact}_{n}(\mathcal{A},T,\omega)$. Moreover an $n$-fold $(\mathcal{A},T)$-factorization $X$ of $\omega$ is injective if and only if it is a direct summand of $\theta^{0}(I^{0}) \oplus \theta^{1}(I^{1})$ for some injective objects $I^{0}$ and $I^{1}$ in $\mathcal{A}$.
\end{lemma}

\begin{proof}
	Let $X$ be any $n$-fold $(\mathcal{A},T)$-factorization $X$ of $\omega$. Then we can take two inflations $a:X^{n-1} \rightarrow I^{0}$, $b:T(X^{0}) \rightarrow I^{1}$ which $I^{0}$, $I^{1}$ are injective in $\mathcal{A}$. By adjunctions $({\rm pr}^{n-1},\theta^{0})$ and $({\rm pr}^{n-1}S,\theta^{1})$, we can obtain two morphisms $a^{'}: X \rightarrow \theta^{0}(I^{0})$, $b^{'}:X \rightarrow \theta^{1}(I^{1})$. Take the pushout of $a^{'}$ and $b^{'}$, we can obtain an inflation $X \rightarrow \theta^{0}(I^{0})\oplus \theta^{1}(I^{1})$. By lemma \ref{lem4} (2), we have $\theta^{0}(I^{0})$ and $\theta^{1}(I^{1})$ are injective. This proves that $\text{Fact}_{n}(\mathcal{A},T,\omega)$ has enough injective objects.
	
	Moreover, if $X$ is injective, the inflation splits. Therefore this yields the final conclusion.
\end{proof}	

\begin{theorem}\label{main2}
	Let $(\mathcal{A},\mathcal{E})$ be a Frobenius exact category. Then so is $({\rm Fact}_{n}(\mathcal{A},T,\omega),\tilde{\mathcal{E}})$.
\end{theorem}	

\begin{proof}
	By Lemma \ref{lem5}, $\text{Fact}_{n}(\mathcal{A},T,\omega)$ has enough projective objects, and projective objects are precisely direct summands of $\theta^{0}(P^{0})\oplus \theta^{1}(P^{1})$ for some projective objects $P^{0}$ and $P^{1}$ in $\mathcal{A}$. Meanwhile, $P^{0}$ and $P^{1}$ are injective in $\mathcal{A}$. By Lemma \ref{lem4} (2), we have $\theta^{0}(P^{0})\oplus \theta^{1}(P^{1})$ is injective and the direct summands of $\theta^{0}(P^{0})\oplus \theta^{1}(P^{1})$ are also injective. Therefore the projective objects in $\text{Fact}_{n}(\mathcal{A},T,\omega)$ are injective.
	
	Dually, $\text{Fact}_{n}(\mathcal{A},T,\omega)$ has enough injective objects, and injective objects are projective.
	
	Consequently, we conclude that $(\text{Fact}_{n}(\mathcal{A},T,\omega),\tilde{\mathcal{E}})$ is a Frobenius exact category.
\end{proof}	

\begin{corollary}
	Let $(\mathcal{A},\mathcal{E})$ be a Frobenius exact category. Then the stable category of the Frobenius exact category $({\rm Fact}_{n}(\mathcal{A},T,\omega),\tilde{\mathcal{E}})$ is a triangulated category.
\end{corollary}

\proof  It follows from Theorem \ref{main2} and \cite[Theorem 2.6]{Ha}.  \qed

\begin{remark}
	\begin{itemize}
		\item [(1)]
	When $n=2$, the conclusion corresponds to \cite[Proposition 4.4]{C}.
	\item [(2)]
	If we change the sign of the differentials in
	the shifted $n$-fold $(\mathcal{A},T)$-factorization $S$, that is, let
	$S=\Sigma$, the conclusions in Section 4 will also hold.
\end{itemize}
\end{remark}

%

\textbf{Yixia Zhang}\\
School of Mathematics and Statistics, Changsha University of Science and Technology, 410114 Changsha, Hunan,  P. R. China\\
E-mail: yxzhangmath@163.com\\[0.3cm]
\textbf{Panyue Zhou}\\
School of Mathematics and Statistics, Changsha University of Science and Technology, 410114 Changsha, Hunan,  P. R. China\\
E-mail: panyuezhou@163.com

\end{document}